\documentclass[12pt,letterpaper,twoside]{article}

\usepackage{amsmath,amssymb}
\usepackage{amsthm}
\usepackage{mathtools}
\usepackage{stmaryrd}
\usepackage{fancyhdr}
\usepackage{times}
\usepackage{mathrsfs} 
\usepackage{titlesec}
\usepackage{enumitem}
\usepackage[pdftex,dvips]{geometry}

\input diagxy.tex

\mathtoolsset{centercolon,mathic}

\titlelabel{\thetitle.\ }
\titleformat*{\section}{\normalsize\bfseries\centering}
\titleformat*{\subsection}{\normalsize\itshape}

\newtheoremstyle{fact}
     {\topsep}
     {\topsep}
     {\slshape}
     {}
     {\bfseries}
     {}
     { }
     {\thmname{#1}\thmnumber{ #2.}\thmnote{ \rm (#3)}}

\newtheoremstyle{mylabel} 
	{\topsep}
	{\topsep}
	{\itshape}
	{}
	{\bfseries}
	{}
	{ }
	{\thmname{#1}\thmnote{ #3}.} 

\newtheorem{theorem}{Theorem}[section]
\newtheorem{Ltheorem}{Theorem}

\newtheorem*{theorem*}{Theorem} 
\newtheorem{lemma}[theorem]{Lemma}
\newtheorem{proposition}[theorem]{Proposition}
\newtheorem{corollary}[theorem]{Corollary}
\newtheorem{problem}{Problem}

\theoremstyle{definition}
\newtheorem{definition}[theorem]{Definition}
\newtheorem{remark}[theorem]{Remark}

\newtheorem*{remark*}{Remark}

\newtheorem*{question*}{Question}
\newtheorem{examples}[theorem]{Examples}
\newtheorem*{examples*}{Examples}
\newtheorem{example}[theorem]{Example}
\newtheorem*{example*}{Example}

\theoremstyle{mylabel}
\newtheorem*{Ltheorem*}{Theorem}

\theoremstyle{fact}

\newtheorem{ftheorem}[theorem]{Theorem}

\setenumerate{topsep=0pt,itemsep=3pt,parsep=0pt,itemindent=0in,
leftmargin=0.275in,labelindent=0pt,itemindent=0pt,
align=left,labelsep=0in,labelwidth=0.275in,label={\rm(\alph*)}}

\def\proofont{\fontseries{bx}\fontshape{sc}\selectfont}
\def\proofname{Proof. }

\newcommand{\pcite}[2]{{\cite[#1]{#2}}}

\newcommand{\homeo}{\operatorname{Homeo}}
\newcommand{\colim}{\operatorname*{colim}}

\makeatletter
\renewenvironment{proof}[1][\proofname]{\par
  \normalfont
  \topsep6\p@\@plus6\p@ \trivlist
  \item[\hskip\labelsep\noindent\proofont #1]\ignorespaces
}{%
  \qed\endtrivlist
}
\makeatother

\makeatletter
\def\@fnsymbol#1{\ifcase#1\or * \or 1 \or 2  \else\@ctrerr\fi\relax}
\let\mytitle\@title

\author{Rafael Dahmen \ and G\'abor Luk\'acs}

\title{Long colimits of topological groups I: Continuous maps and homeomorphisms}

\begin{document}

\makeatletter
\let\mytitle\@title
\chead{\small\itshape R. Dahmen and G. Luk\'acs / Long colimits of topological groups I}
\fancyhead[RO,LE]{\small \thepage}
\makeatother

\maketitle

\def\thanks#1{}

\thispagestyle{empty}

\begin{abstract}
Given a directed family of topological groups, the finest topology on their union making each injection continuous need not be a group topology, because the multiplication may fail to be jointly continuous. 
This begs the question of when the union is a topological group with respect to this topology. If the family is countable, the answer is well known in most cases. We study this question in the context of so-called {\itshape long} families, which are as far as possible from countable ones. As a first step, we present answers to the question for families of group-valued continuous maps and homeomorphism groups, and provide additional examples.
\end{abstract}

\section{Introduction}

Given a directed family $\{G_\alpha\}_{\alpha \in \mathbb{I}}$ of topological groups with closed embeddings as bonding maps, their union $G=\bigcup\limits_{\alpha\in \mathbb I} G_\alpha$ can be equipped with two topologies: the {\itshape colimit space topology} defined as the finest topology $\mathscr{T}$ making each map $G_\alpha \rightarrow G$ continuous, and the {\itshape colimit group topology}, defined as the finest {\itshape group topology} $\mathscr{A}$ making each map $G_\alpha \rightarrow G$ continuous. The former is always finer than the latter, which begs the question of when the two topologies coincide.

We say that $\{G_\alpha\}_{\alpha \in \mathbb{I}}$ satisfies the {\em algebraic colimit property} (briefly, {\itshape ACP}) if $\mathscr{T}=\mathscr{A}$, that is, if the colimit of $\{G_\alpha\}_{\alpha \in \mathbb{I}}$ in the category $\mathsf{Top}$ of topological spaces and continuous maps coincides with the colimit in the category $\mathsf{Grp(Top)}$ of topological groups and their continuous homomorphisms.

The family $\{G_\alpha\}_{\alpha \in \mathbb{I}}$ satisfies {ACP} if and only
if $(G,\mathscr{T})$ is a topological group.
In general, the inversion $(G,\mathscr{T}) \rightarrow (G,\mathscr{T})$ is continuous, and the multiplication
$m\colon (G,\mathscr{T}) \times (G,\mathscr{T}) \rightarrow (G,\mathscr{T})$ is separately continuous, but need not be jointly continuous. Thus, $\{G_\alpha\}_{\alpha \in \mathbb{I}}$ satisfies {ACP} if and only if $m$ is continuous.

It is well known that ACP holds if: (1) the bonding maps are all open; or (2) $I$ is countable and each $G_\alpha$ is locally compact Hausdorff (\cite[Theorem~4.1]{Hirai1} and \cite[Propositions 4.7 and 5.4]{Gloeckner1}). Yamasaki showed that for countable families of metrizable groups, these two are the only cases where ACP holds. 

\begin{ftheorem}[\pcite{Theorem 4}{Yamasaki}] \label{thm:Yamasaki}
Let $\{G_n\}_{n\in\mathbb N}$ be a countable family of metrizable topological groups with closed embeddings as bonding maps. If
\begin{enumerate}

\item
there is $n$ such that $G_n$ is not locally compact, and

\item
for every $n$ there is $m>n$ such that $G_n$ is not open in $G_m$,

\end{enumerate}
then the colimit space topology $\mathscr{T}$ is not a group topology on $G=\bigcup G_n$, that is, $\{G_n\}_{n\in\mathbb N}$ does not satisfy ACP.
\end{ftheorem}

While ACP is well understood for countable families of metrizable groups, Yamasaki's theorem does not hold in the uncountable case. For example, if $\mathbb I=\omega_1$ and the bonding maps are isometric embeddings, then ACP holds (Corollary~\ref{cofin:cor:Lipschitz}). Similarly, if $\mathbb I=\omega_1$ and $G_\alpha = \prod\limits_{\beta < \alpha} M_\beta$, where each $M_\alpha$ is metrizable, then ACP holds, and the resulting group is the $\Sigma$-product of the $M_\beta$ (Corollary~\ref{cofin:cor:sigma}).

Motivated by these examples, the aim of the present paper is to examine directed families that are as far as possible from countable ones. A directed set $(\mathbb{I},\leq)$ is {\em long} if every countable subset of $\mathbb{I}$ has an upper bound in $\mathbb{I}$. It may be tempting to conjecture that all long families of topological groups satisfy ACP; this, however, is not the case (see Example~\ref{cofin:ex:Bisgaard}). Nevertheless, long families provide a number of interesting examples where ACP holds.

We say that a Hausdorff space $X$ is a {\em long space} if the set $\mathscr{K}(X)$ of compact subsets of $X$ ordered by inclusion is long, that is, every $\sigma$-compact subset of $X$ has a compact closure. The product $\mathbb{L}_{\geq 0}:=\omega_1 \times [0,1)$ equipped with order topology generated by the lexicographic order is called the {\itshape Closed Long Ray}. The {\itshape Long Line} $\mathbb{L}$ is obtained by gluing together two copies of the Closed Long Ray, one with the reverse order and one with the usual order, 
at the boundary points $(0,0)$. 
The connected spaces $\mathbb{L}_{\geq 0}$ and $\mathbb{L}$ share many of the properties of the zero-dimensional space $\omega_1$: they are locally compact Hausdorff, first countable, and normal, but neither metrizable nor paracompact. Both $\mathbb{L}_{\geq 0}$ and $\mathbb{L}$ are long spaces, and any cardinal of uncountable cofinality with the order topology is a long space. 

For a topological space $X$ and a topological group $M$, we put $C_{cpt} (X,M)$ for the set
of continuous compactly supported $M$-valued functions on $X$, where the {\itshape support} of 
$f\colon X \rightarrow M$ is  $\operatorname{cl}_X \{ x \in X \mid f(x) \neq e_M\}$.
We equip $C_{cpt} (X,M)$ with pointwise operations and the uniform topology. The next theorem
shows that for a long space $X$, the uniform topology is the ``natural'' one on $C_{cpt} (X,M)$.

\begin{Ltheorem} \label{thm:main:cts} 
Let $X$ be a long space and $M$ be a metrizable group.
For $K \in \mathscr{K}(X)$, let $C_K(X,M)$ denote the subgroup of $C_{cpt} (X,M)$  consisting of functions with support in $K$.
Then the family $\{C_K(X,M)\}_{K\in\mathscr{K}(X)}$ satisfies ACP and
\begin{align}
\colim\limits_{K \in \mathscr{K}(X)}  C_K(X,M) = C_{cpt} (X,M).
\end{align}
\end{Ltheorem}

Recall that a space $X$ is {\itshape pseudocompact} if every continuous real-valued map on $X$ is bounded. 
One says that $X$ is {\itshape $\omega$-bounded} if every countable subset of $X$ is contained in a compact subset.
Long spaces have previously been studied under the names {\itshape strongly $\omega$-bounded} and {\itshape $\sigma C$-bounded} (\cite{Nyikos2007} and \cite{vanMill2013}).
The relationship between a Hausdorff space $X$ being a long space and other compactness-like properties is as follows:
\begin{align}
\mbox{compact } \Longrightarrow
\mbox{ long }  \Longrightarrow
\mbox{ $\omega$-bounded } \Longrightarrow
\mbox{ pseudocompact},
\end{align}
and for metrizable spaces they are all equivalent. Nyikos showed that an $\omega$-bounded space need not be long (\cite{Nyikos2007}). The next theorem shows that the requirement of $X$ being long cannot be omitted from Theorem~\ref{thm:main:cts}.

\begin{Ltheorem} \label{thm:main:neg-cts}
Let $X$ be a locally compact Hausdorff space and $V$ a non-trivial metrizable topological vector space. If $X$ is not pseudocompact, then the following statements are equivalent:

\begin{enumerate}[label={\rm (\roman*)}]

\item
the family $\{C_K(X,V)\}_{K\in\mathscr{K}(X)}$ satisfies ACP;

\item 
$X \cong \mathbb{N}$ (discrete) and $V$ is finite dimensional.

\end{enumerate}
\end{Ltheorem}

The proofs of Theorems~\ref{thm:main:cts} and~\ref{thm:main:neg-cts} are presented in Section~\ref{sect:cts}, and are based on basic properties of long colimits established in Section~\ref{sect:cofin}.

Another interesting way to obtain a family of topological groups indexed by the compact subsets of a space $X$ is considering $\homeo_K(X)$, the homeomorphisms that are the identity outside a given compact set $K \in \mathscr{K}(X)$, equipped with the compact open topology. The notion of support, which plays a central role in Theorem~\ref{thm:main:cts}, can also be defined for homeomorphisms: the {\itshape support} of  $h\colon X\rightarrow X$ is $\operatorname{cl}_X\{x \in X \mid h(x)\neq x\}$. It is not hard to show that $\{\homeo_K(\mathbb{R})\}_{K \in \mathscr{K}(\mathbb{R})}$ does not satisfy ACP. In fact, a more general negative result, in the flavour of  Theorem~\ref{thm:main:neg-cts}, holds.

\begin{Ltheorem} \label{thm:main:neg-homeo}
Let $X$ be a locally compact metrizable space that is not compact, and suppose that $\homeo_{\overline U} (X)$ is not locally compact for every non-empty open set $U$ with $\overline{U}$ compact. Then $\{\homeo_K(X)\}_{K \in \mathscr{K}(X)}$ does not satisfy ACP.
\end{Ltheorem}

The proof of Theorem~\ref{thm:main:neg-homeo} is presented in Section~\ref{sect:homeo-neg}, where we 
establish a partial analogue of Theorem~\ref{thm:main:neg-cts} for homeomorphism groups, and show that a stronger,
more general version of Theorem~\ref{thm:main:neg-homeo} holds.
Notably, the spaces satisfying the conditions of Theorem~\ref{thm:main:neg-homeo} are not long. 
For certain long spaces, the following analogue to Theorem~\ref{thm:main:cts} holds.

For a Tychonoff space $X$, we put $\homeo_{cpt}(X)$  for the set of compactly supported 
homeomorphisms of $X$. Every compactly supported homeomorphism on $X$ can be extended to
a homeomorphism of the Stone-\v{C}ech compactification $\beta X$ of $X$ that is the 
identity on the remainder $\beta X \backslash X$.
 Thus,  the group $\homeo_{cpt}(X)$ of compactly supported homeomorphisms of $X$ is a subgroup of $\homeo(\beta X)$, and we may equip 
$\homeo_{cpt}(X)$  with the compact-open topology induced by $\beta X$. Then $\homeo_K(X)=\homeo_K(\beta X)$ is a topological subgroup of $\homeo_{cpt}(X)$ for every $K \in \mathscr{K}(X)$.

\begin{definition}
A Tychonoff space $X$ has the {\itshape Compactly Supported Homeomorphism Property} ({\itshape CSHP}) if 
\begin{align}
\colim\limits_{K \in \mathscr{K}(X)} \homeo_K(X) = \homeo_{cpt}(X),
\end{align}
where the left-hand side is equipped with the colimit space topology. 
\end{definition}

Clearly, if $X$ has CSHP, then $\{\homeo_K(X)\}_{K\in \mathscr{K}(X)}$ satisfies ACP; however, the converse is false. For example, if $X$ is an infinite discrete set, then each $\homeo_K(X)$ is finite, and so $\{\homeo_K(X)\}_{K\in \mathscr{K}(X)}$ satisfies ACP, but $X$ does not have CSHP (see Corollary~\ref{cshp:cor:discrete}).

\begin{Ltheorem} \label{thm:main:homeo}
Suppose that
\begin{enumerate}

\item
$X=\mathbb{L}_{\geq 0}$; or

\item
$X=\mathbb{L}$; or

\item
$X = \kappa^n \times \lambda_1 \times \cdots \times \lambda_j$, where $\kappa$ is an uncountable regular cardinal,
$\lambda_1,\ldots,\lambda_j$ are successor ordinals smaller than $\kappa$, and $n,j \in \mathbb{N}$.

\end{enumerate}

\noindent
Then $X$ has CSHP.
\end{Ltheorem}

 The proof of Theorem~\ref{thm:main:homeo} is presented in Section~\ref{sect:cshp}, where we study  CSHP. 

Finally, we turn to posing some open problems.

The directed limit of $T_1$-spaces or normal spaces is again $T_1$ or normal, respectively; however, the directed limit of Hausdorff or Tychonoff spaces need not be Hausdorff or Tychonoff (\cite{Herrlich}).

\begin{problem}
Let $\{G_\alpha\}_{\alpha \in \mathbb{I}}$ be a long directed family of topological groups with closed embeddings as bonding maps. 

\begin{enumerate}

\item
Is the colimit group topology $\mathcal{A}$ Hausdorff? 

\item
Is the colimit space topology $\mathcal{T}$ Hausdorff?

\end{enumerate}
\end{problem}

There is a small gap between the negative statement established in Theorem~\ref{thm:main:neg-cts} and the positive one in Theorem~\ref{thm:main:cts}. There are pseudocompact spaces that are not long (\cite{Nyikos2007}).

\begin{problem}
Let $X$ be a pseudocompact Tychonoff space that is not long. Does $\{C_K(X,\mathbb{R}\}_{K\in\mathcal{K}(X)}$ satisfy ACP?
\end{problem}

There is a big gap between the negative statement established in Theorem~\ref{thm:main:neg-homeo} (or its generalization in Section~\ref{sect:homeo-neg}) and the positive one in Theorem~\ref{thm:main:homeo}. There are many long locally compact Tychonoff spaces that are not compact, such as finite powers of the long line or the long ray, but it is not known whether they satisfy CSHP.

\begin{problem}
Characterize the long, Tychonoff, non-compact spaces $X$ that satisfy CSHP.
\end{problem}

\section{Long families and tightness}
\label{sect:cofin}

A directed family of topological spaces $\{X_\alpha\}_{\alpha \in \mathbb{I}}$ 
is called {\em strict} if every bonding map $X_\alpha \rightarrow X_\beta$ is an embedding.
A topology $\mathscr{S}$ on the set-theoretic union $X=\bigcup\limits_{\alpha \in \mathbb{I}} X_\alpha$ is {\em admissible} if it makes each inclusion $X_\alpha \rightarrow (X,\mathscr{S})$ an embedding. If there is an admissible topology on $X$, then the family is strict. Furthermore, it can easily be seen that if the bonding maps are {\itshape closed} (or {\itshape open}) embeddings, the colimit space topology
\begin{align}
\mathscr{T}:=\{U\subseteq X \mid U\cap X_\alpha \text{ is open in } X_\alpha \text{ for all } \alpha\in\mathbb{I}\}
\end{align}
of a family of spaces is admissible.

In this section, we provide a sufficient condition for the colimit space  topology to be the only admissible one (Proposition~\ref{cofin:prop:admissible}), and use it to show that ACP holds for many long families of topological groups. 

Recall that the \emph{tightness} $t(X,\mathscr{S})$ of a topological space $(X,\mathscr{S})$ is the smallest cardinal $\kappa$ such that every point $p$ in the closure of a subset $A\subseteq X$ is in the closure of a subset $C\subseteq A$ with $|C| \leq \kappa$. One says that a space is \emph{$\kappa$-tight} [\emph{countably tight}] if $t(X,\mathscr{S})\leq\kappa$ [if $t(X,\mathscr{S})=\aleph_0$]. For ease of reference, we start
off with a well-known characterization of the tightness.

\begin{lemma} \label{cofin:lemma:tight}
Let $X$ be a topological space and $\kappa$ be a cardinal. The following statements are equivalent:

\begin{enumerate}[label={\rm (\roman*)}]
    \item 
    $t(X) \leq \kappa$;
    
    \item
    $X = \colim\limits_{C\leq X, |C|\leq \kappa} C$ in $\mathsf{Top}$, that is, 
    for every topological space $Y$ and $f: X\rightarrow Y$, if $f_{|C}$ is continuous for every subset $C\subseteq X$ with $|C|\leq\kappa$, then $f$ is continuous on $X$.
    
 \end{enumerate}

\end{lemma}

\begin{proof}
(i) $\Rightarrow$ (ii): Let $f\colon X\rightarrow Y$ be a map such that $f_{|C}$ is continuous for every
subset $C\subseteq X$ with $|C|\leq \kappa$. To show that $f$ is continuous, let $A\subseteq X$ and
$x \in \operatorname{cl} A$. Since $t(X) \leq \kappa$, there is $C\subseteq A$ with $|C|\leq \kappa$ such that $x \in \operatorname{cl} C$. Put $D=C \cup \{x\}$. By our assumption, $f_{|D}$ is continuous, and so
\begin{align}
    f(x) \in f(\operatorname{cl}_D C) \subseteq \operatorname{cl} f(C) \subseteq \operatorname{cl} f(A).
\end{align}
Thus, $f(\operatorname{cl} A) \subseteq \operatorname{cl} f(A)$, as desired.

(ii) $\Rightarrow$ (i): For $A\subseteq X$, put 
$\operatorname{cl}^\kappa (A):= \bigcup \{ \operatorname{cl} C \mid C\subseteq A, |C| \leq \kappa\}$.    
It is easily shown that  $\operatorname{cl}^\kappa$ is a Kuratowski closure operator, and as such
it defines a topology $\mathscr{T}_\kappa$ on $X$ (\cite[1.2.7]{Engel}). The topology $\mathscr{T}_\kappa$ is finer than the original topology $\mathscr{T}$ of $X$, and by (ii), the identity map
$(X,\mathscr{T})\rightarrow (X,\mathscr{T}_\kappa)$ is continuous. Therefore, $\mathscr{T}=\mathscr{T}_\kappa$, and hence $\operatorname{cl}A = \operatorname{cl}^\kappa A$ for 
every $A \subseteq X$.
\end{proof}

\begin{definition}
Given a cardinal $\kappa$, a directed set $(\mathbb{I},\leq)$ is {\em $\kappa$-long}  if every subset $C\subseteq \mathbb{I}$ with $|C| \leq \kappa$ has an upper bound in $\mathbb{I}$. One says that $\mathbb{I}$ is {\em long} if it is $\aleph_0$-long. A directed family of spaces or groups is $\kappa$-long [long] if it is indexed by a $\kappa$-long [long] set.
\end{definition}

A totally ordered set is long if and only if its cofinality is not $\omega$; however, this
equivalence fails for ordered sets that are not totally ordered. For example, 
$\omega \times \omega_1$ with the product order
has cofinality $\omega_1$, but it is not long, because $\omega \times \{0\}$ has no upper bound.

\begin{proposition} \label{cofin:prop:admissible}
Let $\{X_\alpha\}_{\alpha \in \mathbb{I}}$ be a $\kappa$-long family of topological spaces. For an admissible topology $\mathscr{S}$  on $X=\bigcup\limits_{\alpha \in \mathbb{I}} X_\alpha$, the following are equivalent:

\begin{enumerate}[label={\rm (\roman*)}]

\item
$t(X,\mathscr{S}) \leq \kappa$;

\item
$\mathscr{S}=\mathscr{T}$ (the colimit space topology), and $t(X_\alpha,\mathscr{S}) \leq \kappa$ for all $\alpha \in \mathbb{I}$.
\end{enumerate}

\end{proposition}

\begin{proof}
(i) $\Rightarrow$ (ii): One has $\mathscr{S}\subseteq \mathscr{T}$, because the embeddings $X_\alpha \rightarrow (X,\mathscr{S})$ induce a continuous map $(X,\mathscr{T}) \rightarrow (X,\mathscr{S})$. For the reverse inclusion, we  show that
$\iota\colon (X,\mathscr{S}) \rightarrow (X,\mathscr{T})$ is continuous. Let $C\subseteq X$ be such that $|C| \leq \kappa$. Since $\mathbb{I}$ is $\kappa$-long, there is $\alpha \in \mathbb{I}$ such that $C \subseteq X_\alpha$. The subspace topology induced by
$\mathscr{S}$ on $C$ coincides with the topology induced by $X_\alpha$, because $\mathscr{S}$ is admissible. Thus,
\begin{align}
\iota_{|C}\colon (C,\mathscr{S}_{|C}) \rightarrow X_\alpha \rightarrow (X,\mathscr{T})
\end{align}
is continuous. Therefore, $\iota$ is continuous, because $(X,\mathscr{S})$ is $\kappa$-tight. Hence, $\mathscr{S} =\mathscr{T}$. The second statement follows from $X_\alpha$ being a subspace of $(X,\mathscr{S})$.

(ii) $\Rightarrow$ (i): It follows from Lemma~\ref{cofin:lemma:tight} that the colimit of $\kappa$-tight spaces is $\kappa$-tight.
\end{proof}

\begin{corollary} \label{cofin:cor:product}
Let $\{G_\alpha\}_{\alpha \in \mathbb{I}}$ be a $\kappa$-long family of topological groups, and endow 
$G=\bigcup\limits_{\alpha \in \mathbb{I}} G_\alpha$ with the colimit space topology $\mathscr{T}$. If the product $G\times G$ is $\kappa$-tight, then $G$ is a topological group, and $\{G_\alpha\}_{\alpha \in \mathbb{I}}$ satisfies ACP.
\end{corollary}

\begin{proof}
Put $X_\alpha = G_\alpha \times G_\alpha$ for every $\alpha \in \mathbb{I}$. Since $\mathscr{T}\times\mathscr{T}$ is admissible
on $G\times G = \bigcup\limits_{\alpha \in \mathbb{I}} X_\alpha$, by Proposition~\ref{cofin:prop:admissible}, 
$G\times G$ coincides with the colimit of the spaces $\{X_\alpha\}_{\alpha\in\mathbb{I}}$. Therefore, the multiplication
$G\times G \rightarrow G$ is continuous.
\end{proof}


The following construction provides a wealth of examples of countably tight spaces that need not be first countable. Let $\{(Y_j,0_j)\}_{j\in J}$ be a~family of spaces with base points. The {\itshape support} of $y =(y_j) \in \prod\limits_{j\in J} Y_j$ is $\operatorname{supp}(y):=\{ j \in J \mid y_j \neq 0_j\}$.  The {\itshape $\Sigma$-product} of the family is
\begin{align}
\sum\limits_{j \in J} (Y_j,0_j) := \Big\{y \in  \prod\limits_{j\in J} Y_j \mid \operatorname{supp}(y) \mbox{ is countable}\Big\}
\end{align}
equipped with the topology induced by the product topology.

\begin{ftheorem}[{\cite[6.16]{Arhangelskii}}] \label{cofin:thm:sigma}
Let $\{(Y_j,0_j)\}_{j\in J}$ be a family of first countable spaces with base points. Then $\sum\limits_{j \in J} (Y_j,0_j)$ is countably tight.
\end{ftheorem}



\begin{corollary} \label{cofin:cor:sigma} 
Let $\{G_j\}_{j \in J}$ be a family of metrizable topological groups. Put $\mathbb{I}=[J]^{\leq \omega}$ ordered by inclusion, and for $D \in \mathbb{I}$, put $G_D = \prod\limits_{j \in D} G_j \times \prod\limits_{j \notin D} \{e_j\}$ with the product topology. Then $\{G_D\}_{D\in \mathbb{I}}$ satisfies ACP, and $\colim\limits_{D\in \mathbb{I}} G_D =\sum\limits_{j \in J} (G_j,e_j)$.
\end{corollary}

\begin{proof}
Put $G:=\bigcup\limits_{D\in \mathbb{I}} G_D \subseteq \prod\limits_{j\in J} G_j$, and let $\mathscr{S}$ denote the product topology on $G$. By Theorem~\ref{cofin:thm:sigma}, $(G,\mathscr{S})$ is countably tight, and it follows from the construction that $\mathscr{S}$ is admissible. Thus, by Proposition~\ref{cofin:prop:admissible}, $\mathscr{S}$ coincides with the colimit space topology  $\mathscr{T}$, because the indexing set $\mathbb{I}$ is long. Therefore, $(G,\mathscr{T})$ is a topological group, being a topological subgroup of $\prod\limits_{j \in J} G_j$. Hence, $\mathscr{T}$ coincides with the colimit group topology.
\end{proof}

\begin{example}
Observe that $\mathbb{I}$ cannot be replaced with the set of finite subsets of~$J$ in Corollary~\ref{cofin:cor:sigma}. Let $J=\mathbb{N}$, and let each $G_j$ be any metrizable group that is not locally compact. Then $G_D$ is metrizable but not locally compact for every non-empty finite $D\subseteq J$ and the bonding maps are not open, because the $G_j$ are not discrete. Thus, by Theorem~\ref{thm:Yamasaki}, the colimit space topology is not a~group topology.
\end{example}

Theorem~\ref{cofin:thm:sigma} has an interesting consequence for permutation groups too. For a set $J$, let $\operatorname{Sym}(J)$ denote the group of permutations of $J$ equipped with the pointwise topology induced by $J^J$, where $J$ is equipped with the discrete topology. It is well known that $\operatorname{Sym}(J)$  is a topological group. The {\itshape support} of a permutation $\sigma$ is the set
$\operatorname{supp}(\sigma)=\{x \in J \mid \sigma(x)\neq x\}$. Let $\operatorname{Sym}_\omega(J)$ denote the subgroup of $\operatorname{Sym}(J)$ consisting of permutations with countable support.

\begin{corollary} \label{cofin:cor:permutations}
Let $J$ be a set, put $\mathbb{I}=[J]^{\leq \omega}$ ordered by inclusion, and for $D \in \mathbb{I}$, put 
\begin{align}
G_D = \{\sigma \in \operatorname{Sym}(J) \mid \operatorname{supp}(\sigma) \subseteq D\}, 
\end{align}
equipped with the subgroup topology. Then $\{G_D\}_{D\in \mathbb{I}}$ satisfies ACP, 
and $\colim\limits_{D\in \mathbb{I}} G_D = \operatorname{Sym}_\omega(J)$. 
\end{corollary}

\begin{proof}
Consider the family $\{(J,x)\}_{x\in J}$ of spaces with base points. The support of a permutation $\sigma\in\operatorname{Sym}(J)$ coincides with the $\operatorname{supp}(\sigma)$ of $\sigma$ as an element of $J^J$. Thus,
\begin{align}
\operatorname{Sym}_\omega(J) = \operatorname{Sym}(J) \cap \left(\sum\limits_{x \in J} (J,x)\right).
\end{align}
Therefore, by  Theorem~\ref{cofin:thm:sigma}, $\operatorname{Sym}_\omega(J)$ is countably tight. It follows from the construction that 
the topology of $\operatorname{Sym}_\omega(J)$ is admissible, and so, by Proposition~\ref{cofin:prop:admissible}, the topology of
$\operatorname{Sym}_\omega(J)$ coincides with the colimit space topology $\mathscr{T}$, because the indexing set $\mathbb{I}$ is long. Hence,
$(\operatorname{Sym}_\omega(J),\mathscr{T})$ is a topological group, and $\mathscr{T}$ coincides with the colimit group topology.
\end{proof}

We turn now to families of metrizable spaces and groups with well-behaved bonding maps.

\begin{theorem} \label{cofin:thm:Lipschitz}
Let $\{(X_\alpha,d_\alpha)\}_{\alpha \in \mathbb{I}}$ be a strict long family of metrizable spaces. If 
the bonding maps $(X_\alpha,d_\alpha)\rightarrow (X_\beta,d_\beta)$ are Lipschitz with a fixed constant $L>0$ (in particular, if they are isometries), then
\begin{align}
d(x,y):=\limsup\limits_{\alpha \in \mathbb{I}} d_\alpha(x,y)
= \inf\limits_{\alpha \in \mathbb{I}} \sup \{d_\beta(x,y) \mid x,y \in X_\beta \mbox{ and } \beta \geq \alpha\}
\end{align}
is a metric on $X=\bigcup\limits_{\alpha \in \mathbb{I}} X_\alpha$, and it induces the colimit space topology $\mathscr{T}$.
\end{theorem}

In order to prove Theorem~\ref{cofin:thm:Lipschitz}, we need a technical lemma.

\begin{lemma} \label{lemma:cofin:monotone}
Let $(\mathbb{I},\leq)$ be a long directed set and 
$f\colon (\mathbb{I},\leq) \rightarrow (\mathbb{R},\leq)$ a monotone function. Then $f$ is eventually constant, that is,
there is $\beta \in \mathbb{I}$ and $s\in \mathbb{R}$ such that $f(\gamma)=s$ for all $\gamma \geq \beta$.
\end{lemma}

\begin{proof}
By replacing $f$ with $\arctan \circ f$ or $\arctan \circ (- f)$, we may assume without loss of generality that $f$ is non-decreasing and bounded. Put $s = \sup f(\mathbb{I})$. For every $n \in \mathbb{N}\backslash\{0\}$, there is $\beta_n \in \mathbb{I}$ such that
$f(\beta_n) > s - \frac 1 n$. Since $\mathbb{I}$ is a long directed set, there is $\beta \in \mathbb{I}$ such that $\beta \geq \beta_n$ for every $n$.
By monotonicity, for every $\gamma \geq \beta$, 
one has $s \geq f(\gamma) \geq f(\beta_n) > s- \frac 1 n$, and thus $f(\gamma)=s$.
\end{proof}

\begin{proof}[Proof of Theorem~\ref{cofin:thm:Lipschitz}]
We show that: (1) $d$ is a metric; and (2) $d$ induces an admissible topology on $X$. By Proposition~\ref{cofin:prop:admissible}, it will
follow that $d$ induces the colimit space topology on $X$.

For each $\alpha \in \mathbb{I}$ and $x,y\in X_\alpha$, put
\begin{align}
\widetilde{d}_\alpha(x,y) = \sup\limits_{\beta \geq \alpha} d_\beta(x,y).
\end{align}
Then $d_\alpha \leq \widetilde{d}_\alpha \leq L d_\alpha$, and so $\widetilde{d}_\alpha$ is a metric on $X_\alpha$ that is (strongly) equivalent to $d_\alpha$. For $\beta \leq \gamma$ and $x,y\in X_\beta$, one has $\widetilde{d}_\beta(x,y) \geq 
\widetilde{d}_\gamma(x,y)$. Thus, by Lemma~\ref{lemma:cofin:monotone},
\begin{align}
d(x,y) = 
\min\limits_{\beta \in \mathbb{I}} \{\widetilde{d}_\beta(x,y) \mid  x,y \in X_\beta  \},
\end{align}
and for every $x,y\in X$, there is $\beta(x,y)\in \mathbb{I}$ such that $d(x,y) = \widetilde{d}_\gamma(x,y)$ for every
$\gamma \geq \beta(x,y)$. Since $\mathbb{I}$ is long, for every countable subset $C \subseteq X$ there is $\beta(C) \in \mathbb{I}$ such that $d(x,y) = \widetilde{d}_\gamma(x,y)$ for every $x,y\in C$ and  $\gamma \geq \beta(C)$. 
Therefore, $d$ is a metric, and on every countable set $C\subseteq X$, it induces the same topology as $d_\gamma$ for $\gamma$ large enough. Hence, $d$ is admissible, because the maps $X_\alpha \rightarrow X_\gamma$ are embeddings and metric spaces are countably tight.
\end{proof}


\begin{corollary} \label{cofin:cor:Lipschitz}
Let $\{(G_\alpha,d_\alpha)\}_{\alpha \in \mathbb{I}}$ be a strict long family of metrizable topological groups with bonding maps that
are Lipschitz with a fixed constant $L>0$ (in particular, isometries). Then $\{G_\alpha\}_{\alpha \in \mathbb{I}}$ satisfies ACP, and the colimit topology on $G=\bigcup\limits_{\alpha \in \mathbb{I}} G_\alpha$ is generated by the metric
\begin{align}
d(x,y):=\limsup\limits_{\alpha \in \mathbb{I}} d_\alpha(x,y)
= \inf\limits_{\alpha \in \mathbb{I}} \sup \{d_\beta(x,y) \mid x,y \in G_\beta \mbox{ and } \beta \geq \alpha\}.
\end{align}
\end{corollary}

\begin{proof}
By Theorem \ref{cofin:thm:Lipschitz}, the colimit space topology $\mathscr{T}$ on $G$ is generated by the metric $d$, and thus
$\mathscr{T}\times \mathscr{T}$ is metrizable on $G\times G$; in particular, it is countably tight. Therefore, by Corollary~\ref{cofin:cor:product}, the statement follows.
\end{proof}

\begin{example}
The requirement of having a uniform bound for the Lipschitz constants of the bonding cannot be omitted
in Theorem~\ref{cofin:thm:Lipschitz} and Corollary~\ref{cofin:cor:Lipschitz}. 
Let $M$ be a non-trivial metrizable topological group, fix a metric $d \leq 1$ on $M$, and put $G_\alpha:=M^\alpha=\prod\limits_{\gamma<\alpha} M$ for $\alpha < \omega_1$. Since $\omega_1$ is cofinal in $[\omega_1]^{\leq \omega}$, it follows from Corollary~\ref{cofin:cor:sigma} that  $\colim\limits_{\alpha < \omega_1} G_\alpha = \sum\limits_{\alpha < \omega_1} (M,e)$, and thus the colimit space topology is not metrizable. (In fact, it has an uncountable pseudocharacter, and so it cannot be first countable.) 

We construct by transfinite induction a metric $d_\alpha \leq 1$ on each $G_\alpha$ such that for every $\beta < \alpha$, the embedding $\iota_{\beta\alpha}\colon G_\beta \rightarrow G_\alpha$ and the projection $\pi_{\alpha\beta}\colon G_\alpha\rightarrow G_\beta$ are  Lipschitz. Suppose that such $\{d_\gamma\}_{\gamma <\alpha}$ have already been constructed. If $\alpha=\delta+1$ is a successor ordinal, we put
\begin{align}
d_\alpha(x,y):=\max \{d_\delta (x_{< \delta},y_{< \delta}),d(x_\delta,y_\delta)\}
\end{align}
for $x=(x_\gamma)_{\gamma <\alpha},y=(y_\gamma)_{\gamma <\alpha}  \in G_\alpha$. The embedding
$\iota_{\delta\alpha}$ is an isometry, and $\pi_{\alpha\delta}$ is Lipschitz with constant $1$. Thus, by the inductive hypothesis, the composites $\iota_{\beta\alpha} = \iota_{\beta\delta }\iota_{\delta\alpha}$
and $\pi_{\alpha\beta}=\pi_{\alpha\delta}\pi_{\delta\beta}$ are Lipschitz for every $\beta < \alpha$. If $\alpha$ is a limit ordinal, pick a strictly increasing sequence $\alpha_i < \alpha$ such that $\sup\alpha_i=\alpha$. For every $i< \omega$, pick $L_{i} \geq 1$ such that $\iota_{\alpha_n \alpha_i}$ is $L_{i}$-Lipschitz for every $n \leq i$. For $x=(x_\gamma)_{\gamma <\alpha},y=(y_\gamma)_{\gamma <\alpha}  \in G_\alpha$, put
\begin{align}
d_\alpha(x,y):= \sum\limits_{i <\omega} \tfrac{1}{L_i 2^{i+1}} d_{\alpha_i}(x_{< \alpha_i},y_{< \alpha_i}).
\end{align}
In what follows, for a Lipschitz map $f$, we put $\operatorname{Lip}(f)$ for the smallest
Lipschitz constant for $f$.
Let $n < \omega$. The projection $\pi_{\alpha \alpha_n}$ is Lipschitz, because
$d_{\alpha_n}(x_{< \alpha_n},y_{< \alpha_n}) \leq 2^{n+1} L_n d_\alpha(x,y)$. 
To see that the embedding $\iota_{\alpha_n \alpha}$ is Lipschitz too, suppose that $x,y\in G_{\alpha_n}\subseteq G_\alpha$. If  $i < n$, then one has
$d_{\alpha_i}(x_{< \alpha_i},y_{< \alpha_i}) \leq \operatorname{Lip}(\pi_{\alpha_n \alpha_i}) d_{\alpha_n}(x_{< \alpha_n},y_{< \alpha_n})$. On the other hand, if $n \leq i$, then $\iota_{\alpha_n \alpha_i}(x_{< \alpha_n}) = x_{< \alpha_i}$
and $\iota_{\alpha_n \alpha_i}(y_{< \alpha_n}) = y_{< \alpha_i}$, and so
$d_{\alpha_i}(x_{< \alpha_i},y_{< \alpha_i}) \leq L_i d_{\alpha_n}(x_{< \alpha_n},y_{< \alpha_n})$. 
Therefore,
\begin{align}
d_\alpha(x,y) & = \sum\limits_{i <n} \tfrac{1}{L_i 2^{i+1}} d_{\alpha_i}(x_{< \alpha_i},y_{< \alpha_i}) + 
\sum\limits_{i \geq n} \tfrac{1}{L_i 2^{i+1}} d_{\alpha_i}(x_{< \alpha_i},y_{< \alpha_i}) \\
& \leq \sum\limits_{i < n} \tfrac{\operatorname{Lip}(\pi_{\alpha_n \alpha_i})}{L_i 2^{i+1}} d_{\alpha_n}(x_{< \alpha_n},y_{< \alpha_n}) + \sum\limits_{i \geq n} \tfrac{1}{2^{i+1}} d_{\alpha_n}(x_{< \alpha_n},y_{< \alpha_n}) \\
& \leq (\sum\limits_{i < n}\tfrac{\operatorname{Lip}(\pi_{\alpha_n \alpha_i})}{L_i 2^{i+1}} +1)d_{\alpha_n}(x_{< \alpha_n},y_{< \alpha_n}).
\end{align}
This shows that $\iota_{\alpha_n \alpha}$ is Lipschitz. If $\beta < \alpha$, then there is $n < \omega$ such that $\beta \leq \alpha_n$, and by the inductive hypothesis, the composites $\iota_{\beta\alpha} = \iota_{\beta\alpha_n }\iota_{\alpha_n\alpha}$ and $\pi_{\alpha\beta}=\pi_{\alpha\alpha_n}\pi_{\alpha_n\beta}$ are Lipschitz. \qed
\end{example}

Let $\mathsf{TVS}_{\mathbb K}$ denote the category of topological vector spaces over $\mathbb{K}\in \{\mathbb R, \mathbb C\}$ and their continuous homomorphisms, and let $\mathsf{LCTVS}_\mathbb{K}$ denote its full subcategory consisting of locally convex spaces.

\begin{corollary} \label{cofin:cor:Banach}
Let $\{E_\alpha\}_{\alpha \in \mathbb{I}}$ be a long family of Banach spaces over $\mathbb K$ (where $\mathbb{K}\in \{\mathbb R, \mathbb C\}$) with isometries as bonding maps. Then:

\begin{enumerate}

\item
$\{E_\alpha\}_{\alpha \in \mathbb{I}}$  satisfies ACP;

\item
$E=\bigcup\limits_{\alpha\in\mathbb{I}} E_\alpha$ with the colimit space topology $\mathscr{T}$ is a Banach space; and

\item
$(E,\mathscr{T})$ is the colimit of $\{E_\alpha\}_{\alpha \in \mathbb{I}}$ in $\mathsf{TVS}_\mathbb{K}$ and $\mathsf{LCTVS}_\mathbb{K}$.

\end{enumerate}
\end{corollary}

\begin{proof}
Statement (a) follows from Corollary~\ref{cofin:cor:Lipschitz}. By Theorem~\ref{cofin:thm:Lipschitz}, the topology $\mathscr{T}$ is generated by the metric 
\begin{align}
d(x,y)=\min\limits_{\alpha \in \mathbb{I}} \{ d_\alpha(x,y) \mid x ,y \in E_\alpha \} = 
\|x - y\|_\alpha, \mbox{ where } x,y \in E_\alpha.
\end{align}
Thus, the topology $\mathscr{T}$ is generated by the norm
\begin{align}
\|x\| = \|x\|_\alpha, \mbox{ where } x \in E_\alpha.
\end{align}
In particular, $(E,\mathscr{T})$ is a locally convex topological vector space, and (c) follows.
In order to show~(b), we prove that the normed space $(E,\|\cdot\|)$ is complete. To that end, let $\{x_n\} \subseteq E$ be a Cauchy-sequence. Since $\mathbb{I}$ is long, there is $\alpha \in \mathbb{I}$ such that $x_n \in E_\alpha$ for every $n$. Thus, $\{x_n\}$ is a Cauchy-sequence in the Banach space $E_\alpha$, and $x_n \longrightarrow x_0$ for $x_0 \in E_\alpha$. Hence, $x_n \longrightarrow x_0$ in $E$.
\end{proof}

\begin{example}
Let $J$ be an uncountable set, put $\mathbb{I}=[J]^{\leq \omega}$, and let $\mathbb{K} \in \{\mathbb{R}, \mathbb{C}\}$.  By Corollary~\ref{cofin:cor:Banach}, for $1 \leq p < \infty$, the space 
\begin{align}
\ell^p(J,\mathbb{K}) = \bigcup\limits_{D \in \mathbb{I}} \ell^p(D,\mathbb{K})
\end{align}
is the colimit of $\{\ell^p(D,\mathbb{K})\}_{D\in\mathbb{I}}$ in $\mathsf{Top}$, $\mathsf{TVS}_\mathbb{K}$, and $\mathsf{LCTVS}_\mathbb{K}$.
\end{example}

\begin{example} \label{cofin:ex:Bisgaard}
Given a strict long family of topological groups, the colimit space topology need not be a group topology. Let $J$ be an uncountable set, and for every finite subset $F\subseteq J$, put $G_F:=\mathbb{R}^{F}$ with the Euclidean topology.  Put $\mathbb{I}:=[J]^{\leq \omega}$, and for $D \in \mathbb{I}$, put $G_D:=\mathbb{R}^{(D)}$ equipped with the colimit space topology induced by $\{G_F\}_{F \in [D]^{<\omega}}$. Since each $G_D$ is locally compact, $G_D$ is a topological group. (In fact, it can be shown that $G_D$ carries the box topology.) One has
\begin{align}
\colim\limits_{D \in \mathbb{I}}   G_D = \colim\limits_{D \in \mathbb{I}}   \colim\limits_{F \in [D]^{<\omega}} G_F = 
\colim\limits_{F \in [J]^{<\omega}} G_F
\end{align}
where the colimits are taken in $\mathsf{Top}$. Bisgaard showed that addition is not continuous in the colimit space topology generated by the family $\{G_F\}_{F \in [J]^{<\omega}} $ (\cite[Theorem]{Bisgaard}). Thus, $\{G_D\}_{D \in \mathbb{I}}$ does not satisfy ACP.
\end{example}

\section{Continuous maps with compact support}
\label{sect:cts}

In this section, we prove Theorems~\ref{thm:main:cts} and~\ref{thm:main:neg-cts} concerning groups of continuous functions with compact support. We prove Theorem~\ref{thm:main:cts} by establishing the following more elaborate statement.

\begin{Ltheorem*}[\ref{thm:main:cts}$\empty^\prime$]
Let $X$ be a long space and $M$ be a metrizable group.
For $K \in \mathscr{K}(X)$, let $C_K(X,M)$ denote the subgroup of $C_{cpt} (X,M)$  consisting of functions with support in $K$. Then the family $\{C_K(X,M)\}_{K\in\mathscr{K}(X)}$ satisfies ACP, and
\begin{align}
\colim\limits_{K \in \mathscr{K}(X)}  C_K(X,M) = C_{cpt} (X,M).
\end{align}
Furthermore, if $M$ is a Banach space over  $\mathbb K$ (where $\mathbb{K}\in \{\mathbb R, \mathbb C\}$), then $C_{cpt} (X,M)$ is a Banach space, and it is the colimit of $\{C_K(X,M)\}_{K\in\mathscr{K}(X)}$ in $\mathsf{TVS}_\mathbb{K}$ and $\mathsf{LCTVS}_\mathbb{K}$.
\end{Ltheorem*}

\begin{proof}
Let $\rho$ be a metric on $M$.
For each $K \in \mathscr{K}(X)$, equip $C_K(X,M)$  with the sup-metric $d_K$ on $K$, defined by
$d_K(f_1,f_2):= \sup\limits_{x\in K} \rho(f_1(x),f_2(x))$. The bonding maps are isometries, the family $\{(C_K(X,M),d_K)\}_{K\in\mathscr{K}(X)}$ is strict, and since $X$ is a long space, the family is also long. Thus, by Corollary~\ref{cofin:cor:Lipschitz}, the family satisfies ACP, and the colimit topology on $C_{cpt} (X,M)= \bigcup\limits_{K\in\mathscr{K}(X)} C_K(X,M)$ is generated by the metric
\begin{align}
d(f_1,f_2):=\limsup\limits_{K\in\mathscr{K}(X)} d_K (f_1,f_2) = 
\limsup\limits_{K\in\mathscr{K}(X)} (\sup\limits_{x\in K} \rho(f_1(x),f_2(x))) = \sup\limits_{x\in X} \rho(f_1(x),f_2(x)).
\end{align}
The last statement follows by Corollary~\ref{cofin:cor:Banach}.
\end{proof}

\begin{examples} Let $\mathbb{K} \in \{\mathbb{R}, \mathbb{C}\}$. 

\begin{enumerate}

\item
For an ordinal $\alpha$ of uncountable cofinality, $C_{cpt}(\alpha,\mathbb{K}) = \colim\limits_{\beta< \alpha} C_{[0,\beta]}(\alpha,\mathbb{K})=\colim\limits_{\beta< \alpha} C([0,\beta],\mathbb{K})$.

\item
$C_{cpt}(\mathbb{L}_{\geq 0},\mathbb{K}) = \colim\limits_{\beta< \omega_1} C_{[0,\beta]}(\mathbb{L}_{\geq 0},\mathbb{K})$ ($\omega_1$ is identified with $\omega_1 \times\{0\} \subseteq \mathbb{L}_{\geq 0}$).

\end{enumerate}
\end{examples}

We establish Theorem~\ref{thm:main:neg-cts} by first proving a stronger result using the following observation.

\begin{remark} \label{remark:cts:acp}
Let $\{G_\alpha\}_{\alpha \in \mathbb{I}}$ be a directed family of topological groups that satisfies ACP, and let $H$ be a closed subgroup of $\colim\limits_{\alpha\in\mathbb{I}} G_\alpha$. Then 
$H=\colim\limits_{\alpha\in\mathbb{I}} (G_\alpha\cap H)$ (\cite[Proposition 2.4.15]{Engel}), and thus
$\{G_\alpha \cap H\}_{\alpha \in \mathbb{I}}$ also satisfies ACP.
\end{remark}

\begin{Ltheorem*}[\ref{thm:main:neg-cts}$\empty^\prime$]
Let $X$ be a locally compact Hausdorff space and $V$ a non-trivial metrizable topological vector space. If $X$ contains a regular closed $\sigma$-compact non-compact subset, then the following statements are equivalent:

\begin{enumerate}[label={\rm (\roman*)}]

\item
the family $\{C_K(X,V)\}_{K\in\mathscr{K}(X)}$ satisfies ACP;

\item 
$X \cong \mathbb{N}$ (discrete) and $V$ is finite dimensional.

\end{enumerate}
\end{Ltheorem*}

\begin{proof}
(ii) $\Rightarrow$ (i):
Since $X$ is discrete and countable,  $\mathscr{K}(X)=[X]^{<\omega}$ is countable. Since $V$ is finite dimensional, the topological vector space $C_K(X,V) \cong V^{|K|}$ is finite dimensional, and hence locally compact for every $K \in \mathcal{K}(X)$ (cf.~\cite[Theorem 3.6]{Schaefer}). Thus, $\{C_K(X,V)\}_{K\in\mathscr{K}(X)}$, being a countable family of locally compact groups, satisfies ACP (\cite[Theorem~4.1]{Hirai1} and \cite[Propositions 4.7 and 5.4]{Gloeckner1}).

(i) $\Rightarrow$ (ii): Since $\{C_K(X,V)\}_{K\in\mathscr{K}(X)}$ satisfies ACP, by Remark~\ref{remark:cts:acp},
for every closed subgroup $H$ of $G:=\colim\limits_{K\in\mathscr{K}(X)} C_K(X,V)$, one has
\begin{align} \label{neg-cts:eq:H1}
H = \colim\limits_{K\in\mathscr{K}(X)} (C_K(X,V)\cap H),
\end{align}
and so $\{C_K(X,V)\cap H\}_{K\in\mathscr{K}(X)}$ satisfies ACP. In particular, if $A$ is a closed
subspace of $X$ and $H$ is the (closed) subgroup of $G$ consisting of functions whose support is contained in $A$, then $C_K(X,V)\cap H = C_{K\cap A}(X,V)$, and thus $\{C_K(X,V)\}_{K\in\mathscr{K}(A)}$ satisfies ACP.

In order to prove that $X$ is discrete, let $U\subseteq X$ be an open subset such that  $\operatorname{cl}_X U$ is compact. We show that $U$ is finite. Let $A \subseteq X$ be a regular closed $\sigma$-compact non-compact subspace of $X$. Then $A$ contains a countable increasing family $\{K_n\}_{n=1}^\infty$ of compact subsets such that $A=\bigcup\limits_{n=1}^\infty K_n$. Without loss of generality, we may assume that $\operatorname{cl}_X U \subseteq K_1$.  By enlarging the $K_n$ if necessary, we may also assume that $K_n \subseteq \operatorname{int}_A K_{n+1}$ for every $n$. Then $\{K_n\}_{n=1}^\infty$ is cofinal in $\mathscr{K}(A)$, that is, every compact subset of $A$ is contained in one of the $K_n$. Thus, by (\ref{neg-cts:eq:H1}), 
\begin{align}
H = \colim\limits_{n } C_{K_n}(X,V),
\end{align}
and in particular, addition is continuous in the colimit space topology on the right-hand side. Therefore, since each $C_{K_n}(X,V)$ is metrizable, by Theorem~\ref{thm:Yamasaki}, 

\begin{enumerate}

\item
$C_{K_n}(X,V)$ is locally compact for every $n$, or

\item
there is $n$ such that $C_{K_n}(X,V)$ is an open subgroup of $C_{K_{m}}(X,V)$ for every $m>n$.

\end{enumerate}
Since $A$ is regular closed and not compact, $\operatorname{int}_X A\backslash K_n$ is not empty for every $n$. In particular, there is a continuous non-zero $f\colon X\rightarrow \mathbb{R}$ with compact support $K$ such that
$K \subseteq \operatorname{int}_X A\backslash K_n$. Since $\{K_n\}_{n=1}^\infty$ is cofinal in $\mathscr{K}(A)$,
there is $m$ such that $K\subseteq K_m$. Thus, $C_{K_{n}}(X,\mathbb{R}) \subsetneq C_{K_{m}}(X,\mathbb{R})$,
and it follows that $C_{K_{n}}(X,V) \subsetneq C_{K_{m}}(X,V)$, because $V$ is non-trivial. Consequently, $C_{K_{n}}(X,V)$ cannot be an open subgroup of the (path) connected group $C_{K_{m}}(X,V)$. 

Therefore, (b) is not possible, and (a) holds: $C_{K_n}(X,V)$ is locally compact for every $n$. Since $C_{K_n}(X,V)$ is a topological vector space, it is locally compact if and only if it is finite dimensional (cf.~\cite[Theorem 3.6]{Schaefer}). Hence, (a) implies that each $K_n$ is finite and $V$ is finite dimensional. In particular, $U \subseteq K_1$ is finite, which shows that $X$ is discrete.


Since $X$ contains a non-compact subset $A$, it cannot be finite, and so it remains to show that $X$ is countable. Fix a one-dimensional subspace $W$ of $V$, and let $H$ denote the (closed) subgroup of $G$ consisting of functions with image in $W$. Then $H\cap C_K(X,V) = C_K(X,W) = \mathbb{R}^{|K|}$ for every $K \in\mathscr{K}(X)$, and by (\ref{neg-cts:eq:H1}),
$\{\mathbb{R}^{|K|}\}_{K \in\mathscr{K}(X)}$ satisfies ACP. Therefore, by Example~\ref{cofin:ex:Bisgaard}, $X$ is countable.
\end{proof}

Theorem~\ref{thm:main:neg-cts} follows from Theorem~\ref{thm:main:neg-cts}$\empty^\prime$ and the following lemma.

\begin{lemma}
If $X$ is a locally compact Hausdorff space and $X$ is not pseudocompact, then $X$ contains a regular closed $\sigma$-compact non-pseudocompact subset.
\end{lemma}

\begin{proof}
Let $f\colon X \rightarrow (0,\infty)$ be a continuous unbounded map. Pick $x_0 \in X$. For every $n >0$, pick $x_{n+1}\in X$ such that $f(x_{n+1}) > f(x_n)+2$. Since $f$ is continuous, for every $n\in\mathbb{N}$ there is a neighborhood $U_n$ of $x_n$ such that $\overline U_n$ is compact and $f(U_n)\subseteq (f(x_n)-1,f(x_n)+1)$. To show that $\{U_n\}_{n\in \mathbb{N}}$ is locally finite, let $x \in X$, and pick a neighborhood $V$ of $x$ such that $f(V)\subseteq (f(x)-1,f(x)+1)$. It follows from the choice of the $\{x_n\}_{n\in \mathbb{N}}$ that $V$ meets at most two members of the family $\{U_n\}_{n\in \mathbb{N}}$. Thus,
\begin{align}
A:=\bigcup\limits_{n\in\mathbb{N}} \overline{U_n}  = \overline{\bigcup\limits_{n\in\mathbb{N}} U_n}
\end{align}
is regular closed (\cite[1.1.11]{Engel}) and $\sigma$-compact. The set $A$ is not pseudocompact, because $f_{|A}$ is unbounded.
\end{proof}

Theorems~\ref{thm:main:neg-cts} and~\ref{thm:main:neg-cts}$\empty^\prime$ yield the following corollary.

\begin{corollary}
Let $X$ be a non-discrete locally compact Hausdorff space, and let $V$ be a non-trivial metrizable topological vector space. If

\begin{enumerate}

\item
$X$ is $\sigma$-compact and non-compact; or

\item
$X$ is not pseudocompact, 

\end{enumerate}
then the family $\{C_K(X,V)\}_{K\in\mathscr{K}(X)}$ does not satisfy ACP.
\qed
\end{corollary}

\section{Homeomorphisms with compact support: negative results}
\label{sect:homeo-neg}

For a locally  compact Hausdorff space, the homeomorphism group $\homeo(X)$ of $X$ equipped with the compact-open topology
need not be a topological group; however, if $X$ is compact or locally connected, then the compact-open topology is a group topology on $\homeo(X)$ (\cite[Theorem 4]{Arens}). For a compact subset $K$ of $X$, let $\homeo_K(X)$ denote the group of homeomorphisms of $X$ that fix every point outside $K$ (that is, $h(x)=x$ for all $x \in X \backslash K$), and equip it with the compact-open topology. Then $\homeo_K(X)$ is topologically isomorphic to a subgroup of $\homeo(K)$, where the latter is equipped with the compact-open topology, and as such,  $\homeo_K(X)$ is a topological group.

In this section, we consider when the family $\{\homeo_K(X)\}_{K \in \mathscr{K}(X)}$ satisfies ACP. We start off by proving a partial analogue of Theorem~\ref{thm:main:neg-cts}.

\begin{theorem} \label{homeo:thm:neg-acp}
Let $X$ be a locally compact Hausdorff space, and let $A\subseteq X$ be a regular closed, $\sigma$-compact, non-compact, metrizable subset. If $\{\homeo_K (X)\}_{K \in\mathscr{K}(X)}$ satisfies ACP, then at least one of the following holds:

\begin{enumerate}[label={\rm (\roman*)}]

\item
$\homeo_M (X)$ is locally compact for every compact metrizable subset $M \subseteq X$, or

\item
there is a regular closed compact $K_0\subseteq A$ such that:

\begin{enumerate}[label={\rm (\arabic*)}]

\item
for all $K\in \mathscr{K}(A)$ with $K_0 \subseteq K$, the subset $\homeo_{K_0} (X)$ is open in $\homeo_{K} (X)$; and

\item
for all $C \in \mathscr{K}(\overline{A\backslash K_0})$, the group $\homeo_C(X)$ is discrete.

\end{enumerate}

\end{enumerate}
\end{theorem}

\begin{proof}
We show that if (i) fails, then (ii) holds. To that end, let $M_0$ be a compact metrizable subset of $X$ such that $\homeo_{M_0}(X)$ is not locally compact. By replacing $M_0$ with $\operatorname{cl} (\operatorname{int} M_0)$ if necessary, we may assume that $M_0$ is regular closed. Furthermore, by replacing $A$ with $A \cup M_0$ if necessary, we may assume that $M_0 \subseteq A$.

Since $A$ is $\sigma$-compact, it contains a countable increasing family $\{K_n\}_{n=1}^\infty$ of compact subsets such that $A=\bigcup\limits_{n=1}^\infty K_n$. By enlarging the $K_n$ if necessary, we may assume that $M_0 \subseteq K_n \subseteq \operatorname{int}_A K_{n+1}$ for every $n$. Then $\{K_n\}_{n=1}^\infty$ is cofinal in $\mathscr{K}(A)$, that is, every compact subset of $A$ is contained in one of the $K_n$.

Since $\{\homeo_K (X)\}_{K \in\mathscr{K}(X)}$ satisfies ACP, by Remark~\ref{remark:cts:acp}, for every closed subgroup $H$ of $G:=\colim\limits_{K \in \mathscr{K}(X)} \homeo_K (X)$, one has
\begin{align} \label{neg-homeo:eq:H1}
H= \colim\limits_{K \in \mathscr{K}(X)} (\homeo_K (X)\cap H),
\end{align}
and so $\{\homeo_K (X) \cap H\}_{K \in\mathscr{K}(X)}$ satisfies ACP. In particular, if $A$ is a closed subspace of $X$ and $H$ is the (closed) subgroup of $G$ consisting of homeomorphisms whose support is contained in $A$, then $\homeo_K (X) \cap H = \homeo_{K\cap A} (X)$, and thus $\{\homeo_K (X)\}_{K \in\mathscr{K}(A)}$ satisfies ACP. By (\ref{neg-homeo:eq:H1}), 
\begin{align}
H= \colim\limits_{n} \homeo_{K_n} (X),
\end{align}
and in particular, the group multiplication is continuous in the colimit space topology on the right-hand side. Therefore, since each $\homeo_{K_n}(X)$ is metrizable, by Theorem~\ref{thm:Yamasaki},
\begin{enumerate}

\item
$\homeo_{K_n}(X)$ is locally compact for every $n$, or

\item
there is $n_0$ such that $\homeo_{K_{n_0}}(X)$ is an open subgroup of
$\homeo_{K_m}(X)$ for every $m > n_0$.

\end{enumerate}
By our assumption, $\homeo_{M_0}(X)$ is not locally compact, and $M_0 \subseteq K_n$ for every $n$. So, (a) is not possible, and (b) holds. Put $K_0:=\operatorname{cl}_X (\operatorname{int}_X K_{n_0})$. Observe that $\homeo_{K_0}(X) = \homeo_{K_{n_0}}(X)$. We show that (ii) holds for $K_0$.

To show (ii)(1), let $K\in \mathscr{K}(A)$ such that $K_0 \subseteq K$. Since $\{K_n\}_{n=1}^\infty$ is cofinal in $\mathscr{K}(A)$, there is $m > n_0$ such that $K \subseteq K_m$. Thus, $\homeo_{K_0}(X)$ is open in $\homeo_{K_m}(X)$, and in particular, $\homeo_{K_0}(X)$ is open in $\homeo_{K}(X)$. 

To show (ii)(2), let $C\in \mathscr{K}(\overline{A\backslash K_0})$, and put $L:=K_0 \cup C$. Since $\operatorname{int}_X K_0$ and $\operatorname{int}_X C$ are disjoint, $\homeo_{K_0}(X)$ commutes with $\homeo_{C}(X)$, and the multiplication map
\begin{align}
\homeo_{K_0}(X) \times \homeo_{C}(X) \rightarrow \homeo_L(X)
\end{align}
is an embedding of topological groups. By (ii)(1), $\homeo_{K_0}(X)$ is open in $\homeo_L(X)$. Thus, $\homeo_{K_0}(X)$ is open in $\homeo_{K_0}(X) \times \homeo_{C}(X)$, and therefore $\homeo_{C}(X)$ is discrete.
\end{proof}

We establish Theorem~\ref{thm:main:neg-homeo} by proving a stronger result.

\begin{Ltheorem*}[\ref{thm:main:neg-homeo}$\empty^\prime$]
Let $X$ be a locally compact Hausdorff space that is not compact. Each of the following statements imply the subsequent ones:

\begin{enumerate}[label={\rm(\roman*)}]

\item
$X$ is metrizable, locally Euclidean, and has no isolated points;

\item
$X$ is metrizable and $\homeo_{\overline U}(X)$ is not locally compact for every non-empty open set $U$ such that $\overline U$ is compact;

\item
$X$ is normal, locally metrizable, and contains

\begin{enumerate}[label={\rm (\arabic*)}]

\item
a compact subset $M_0$ such that $\homeo_{M_0}(X)$ is not locally compact; and

\item
an infinite discrete closed set $T$ such that for every  $t \in T$ and every neighborhood $V$ of $t$ with compact closure, $\homeo_{\overline V}(X)$ is not discrete;

\end{enumerate}

\item
$X$ contains 

\begin{enumerate}[label={\rm (\arabic*)}]

\item
a compact metrizable subset $M_0$ such that $\homeo_{M_0}(X)$ is not locally compact; and

\item
a regular closed, $\sigma$-compact, and metrizable subset $A$ whose interior contains an infinite discrete closed set $T$ such that for every $t \in T$ and every neighborhood $V$ of $t$ with compact closure, $\homeo_{\overline V}(X)$ is not discrete; 

\end{enumerate}

\item
$\{\homeo_K (X)\}_{K \in\mathscr{K}(X)}$ does not satisfy ACP.

\end{enumerate}
\end{Ltheorem*}

\begin{proof}
(i) $\Rightarrow$ (ii): Let $U$ be a non-empty open set whose closure in $X$ is compact. Since $X$ is locally Euclidean, $U$ contains a non-empty open set $V$ such that there is  a homeomorphism $h\colon \mathbb{R}^n\rightarrow V$ for some $n\in \mathbb{N}$. Put $W:=h((-1,1)^n)$. Observe that $\overline{W}=h([-1,1]^n)$ and one has $\homeo_{\overline W}(X) \cong \homeo_{[-1,1]^n}(\mathbb{R}^n)$.
The group $\homeo_{[-1,1]^n}(\mathbb{R}^n)$ is not locally compact, because it contains a copy of $\homeo_{[-1,1]}(\mathbb{R})$ as a closed subgroup, and the latter is known to be non-locally compact (it is not even complete with respect to its left uniformity; see, for example \cite{dieudonne1944}). Since  $\homeo_{\overline W}(X) \subseteq \homeo_{\overline U}(X)$, it follows that $\homeo_{\overline U}(X)$ is not locally compact.

(ii) $\Rightarrow$ (iii): Since $X$ is metrizable but not compact, it contains an infinite closed discrete set $T$. 

(iii) $\Rightarrow$ (iv): Since every compact subset of a locally metrizable space is metrizable, (iii)(1) implies (iv)(1). Without loss of generality, we may assume that $T=\{t_0,t_1,\ldots\}$ is countably infinite. In order to show (iv)(2), we construct the set $A$ around $T$.

Let $\{U_i\}_{i \in \mathbb{N}}$ be a family of open sets in $X$ such that $t_i \in U_i$  and $\overline U_i$ is compact for every $i\in \mathbb{N}$, and $\overline U_i \cap \overline U_j = \emptyset$ for $i\neq j$. (Such a family exists, because $T$ is discrete and countable.) Put $F:=X \backslash \bigcup\limits_{i\in\mathbb{N}} U_i$. Since $F$ and $T$ are closed disjoint sets in the normal space $X$, there are disjoint open sets $U$ and $V$ such that $T\subseteq U$ and $F \subseteq V$. Put $W_i:=U_i \cap U$. Then $t_i \in W_i$ and $\overline{W_i}$ is compact, and $\{W_i\}_{i \in \mathbb{N}}$ is locally finite. Thus, 
\begin{align}
A=\bigcup\limits_{i\in\mathbb{N}} \overline{W_i} = \overline{\bigcup\limits_{i\in\mathbb{N}} W_i}
\end{align}
is regular closed (\cite[1.1.11]{Engel}), $\sigma$-compact, and its interior contains $T$. Finally, $A$ is metrizable, because it is locally metrizable and Lindel\"of.

(iv) $\Rightarrow$ (v): Assume that  $\{\homeo_K (X)\}_{K \in\mathscr{K}(X)}$ satisfies ACP. Then, by Theorem~\ref{homeo:thm:neg-acp}, either $\homeo_{M_0}(X)$ is locally compact, or $A$ contains a (regular) closed compact subset $K_0$ such that $\homeo_C(X)$ is discrete for every $C \in \mathscr{K}(\overline{A\backslash K_0})$. By (iv)(1), the first case is not possible, so suppose that the second case holds. Since $T$ is infinite closed and discrete, $T \not\subseteq K_0$. Thus, there is $t_0 \in T$ such that $t_0 \notin K_0$.  Let $V$ be a neighborhood of $t_0$ such that $V \subseteq A\backslash K_0$ and $\operatorname{cl}_X{V}$ is compact, and put $C:= \operatorname{cl}_X{V}$. Then, by (iv)(2),  $\homeo_C(X)$ is not discrete. This contradiction shows that $\{\homeo_K (X)\}_{K \in\mathscr{K}(X)}$ does not satisfy ACP.
\end{proof}

\section{The Compactly Supported Homeomorphism Property (CSHP)}

\label{sect:cshp}

Our goal in this section is to prove  Theorem~\ref{thm:main:homeo}.

\begin{Ltheorem*}[\ref{thm:main:homeo}]
Suppose that
\begin{enumerate}

\item
$X=\mathbb{L}_{\geq 0}$; or

\item
$X=\mathbb{L}$; or

\item
$X = \kappa^n \times \lambda_1 \times \cdots \times \lambda_j$, where $\kappa$ is an uncountable regular cardinal, $\lambda_1,\ldots,\lambda_j$ are successor ordinals smaller than $\kappa$, and $n,j \in \mathbb{N}$.

\end{enumerate}

\noindent
Then $X$ has CSHP.
\end{Ltheorem*}

We prove  parts (a) and (b) of Theorem~\ref{thm:main:homeo} by showing that the compact-open topology on $\homeo_{cpt}(\mathbb{L})$ is countably tight and coincides with the compact-open topology induced by $\beta X$, and then applying Proposition~\ref{cofin:prop:admissible}.

\begin{proposition} \label{cshp:prop:longlineray}
Let $X=\mathbb{L}_{\geq 0}$ or $X=\mathbb{L}$. Then
$\homeo_{cpt}(X)$ equipped with the compact-open topology is countably tight.
\end{proposition}

\begin{proof}
The statement for $X=\mathbb{L}_{\geq 0}$ follows from the statement for the case $X=\mathbb{L}$, because
$\homeo_{cpt}(\mathbb{L}_{\geq 0})$ equipped with the compact-open topology is a subspace of $\homeo_{cpt}(\mathbb{L})$ equipped with the compact-open topology.

Let $X=\mathbb{L}$,  $A\subseteq \homeo_{cpt}(\mathbb{L})$, and suppose that $f_0 \in \operatorname{cl} A$. We construct a countable subset $C\subseteq A$ such that $f_0\in \operatorname{cl} C$. Since composition is continuous in the compact-open topology, without loss of generality, we may assume that $f_0=\operatorname{id}_\mathbb{L}$.

For $\alpha < \omega_1$, put $K_\alpha=[-(\alpha,0), (\alpha,0)]$ and $U_\alpha=(-(\alpha+1,0), (\alpha+1,0))$. The set
\begin{align}
    W_\alpha:=\{h \in \homeo_{cpt}(\mathbb{L}) \mid h(K_\alpha) \subseteq U_\alpha\}
\end{align}
is an open neighborhood of $\operatorname{id}_\mathbb{L}$ in the compact open topology, and thus $\operatorname{id}_\mathbb{L} \in \operatorname{cl} (W_\alpha \cap A)$. Since $K_\alpha$ is compact and $U_\alpha$ is metrizable, the space $\mathscr{C}(K_\alpha,U_\alpha)$ is metrizable in the compact-open topology. Consequently, there is a sequence $\{f_n^{(\alpha)}\} \subseteq W_\alpha \cap A$ such that 
$f_{n|K_\alpha}^{(\alpha)} \longrightarrow \operatorname{id}_{K_\alpha}$ uniformly.

We construct now an increasing sequence of ordinals $\{\alpha_m\} \subseteq \omega_1$ such that $\operatorname{supp} f_{n}^{(\alpha_m)} \subseteq K_{\alpha_{m+1}}$ for every $n,m \in \mathbb{N}$. Pick an arbitrary $\alpha_0 < \omega_1$. Suppose that $\alpha_m$ has already been constructed. Since $\mathbb{L}$ is a long space, the countable family of compact subsets $\{\operatorname{supp} f_{n}^{(\alpha_m)}\}_{n\in \mathbb{N}}$ has an upper bound; in particular, there is $\beta < \omega_1$ such that $\operatorname{supp} f_{n}^{(\alpha_m)} \subseteq K_\beta$ for every $n\in \mathbb{N}$. Put $\alpha_{m+1}: = \max(\beta,\alpha_m+1)$.

Since $(\omega_1,\leq)$ is long, $\gamma:=\sup\limits_{m\in\mathbb{N}} \alpha_m$ exists in $\omega_1$, and $\operatorname{supp} f_{n}^{(\alpha_m)} \subseteq K_\gamma$ for every $n,m \in \mathbb{N}$. Put
\begin{align}
    C:=\{f_{n}^{(\alpha_m)}\mid n,m \in \mathbb{N}\} \subseteq A.
\end{align}
We show that $\operatorname{id}_\mathbb{L}\in \operatorname{cl} C$. Observe that $C\subseteq \homeo_{K_\gamma} (\mathbb{L})$. Pick an order-preserving homeomorphism $\varphi\colon K_\gamma \rightarrow [-1,1]$.
It suffices to show that $\operatorname{id}_{[-1,1]} \in \operatorname{cl}(\varphi C  \varphi^{-1})$ in $\homeo([-1,1])$.


Let $\varepsilon > 0$. There is $m \in \mathbb{N}$ such that $a:=\varphi(-(\alpha_m,0)) < -1 + \frac \varepsilon 2$ and $b:=\varphi(\alpha_m,0) > 1 - \frac \varepsilon 2$. By our construction, $f_{n|K_{\alpha_m}}^{(\alpha_m)} \longrightarrow \operatorname{id}_{K_{\alpha_m}}$ uniformly, and so $\varphi f_{n|K_{\alpha_m}}^{(\alpha_m)} \varphi^{-1} \longrightarrow \operatorname{id}_{[a,b]}$ uniformly. Thus, there is $n \in \mathbb{N}$ such that $| \varphi f_{n|K_{\alpha_m}}^{(\alpha_m)} \varphi^{-1} (x) -x | < \frac \varepsilon 2$ for every  $x \in [a,b]$. Since every compactly supported homeomorphism of $\mathbb{L}$ is order-preserving, for every $x \in [b,1]$, one has
\begin{align}
    1 \geq \varphi f_{n|K_{\gamma}}^{(\alpha_m)} \varphi^{-1} (x) \geq  \varphi f_{n|K_{\gamma}}^{(\alpha_m)} \varphi^{-1} (b) > b - \frac \varepsilon 2 > 1 - \varepsilon,
\end{align}
and similarly, for $x\in [-1,a]$, one has $-1 \leq f_{n|K_{\gamma}}^{(\alpha_m)} \varphi^{-1} (x) < -1 + \varepsilon$. Hence, $| \varphi f_{n|K_{\alpha_m}}^{(\gamma)} \varphi^{-1} (x) -x | < \varepsilon$ for every $x \in [-1,1]$. 
\end{proof}


\begin{proof}[Proof of Theorem~\ref{thm:main:homeo}(a)-(b).]
Let $X=\mathbb{L}_{\geq 0}$ or $X=\mathbb{L}$, and let $\alpha X$ denote its one-point compactification. Since $X$ is locally connected, the compact-open topology on $\homeo(X)$, and thus on $\homeo_{cpt}(X)$, coincides with the compact-open topology induced by $\alpha X$  (\cite[Theorems~1 and~4]{Arens}). The compact-open topology is admissible with respect to the family $\{\homeo_K(X)\}_{K\in\mathscr{K}(X)}$, and  by Proposition~\ref{cshp:prop:longlineray}, $\homeo_{cpt}(X)$ is countably tight in the compact-open topology. Thus, by Proposition~\ref{cofin:prop:admissible},
\begin{align}
\colim\limits_{K \in \mathscr{K}(X)} \homeo_K(X) = \homeo_{cpt}(X),
\end{align}
where the left-hand side is equipped with the colimit space topology and the right-hand side is equipped with the compact-open topology. Since $\homeo_{cpt}(X)$ is a topological group, this shows that the family $\{\homeo_K(X)\}_{K\in\mathscr{K}(X)}$ satisfies ACP. Therefore, it remains to show that the compact-open topology on $\homeo_{cpt}(X)$  coincides  with the compact-open topology induced by $\beta X$.

(a) Since $\beta \mathbb{L}_{\geq 0} = \alpha \mathbb{L}_{\geq 0}$ (\cite[16H6]{GilJer}), the statement follows for $X=\mathbb{L}_{\geq 0}$.

(b) It is easily seen that $\beta \mathbb{L} = \mathbb{L} \cup \{-\infty, \infty\}$. (Since $\mathbb{L}$ is a pushout of 
two copies of $\mathbb{L}_{\geq 0}$ over $\{0\}$ and $\beta$ preserves pushouts, $\beta \mathbb{L}$ is the pushout of two copies of $\beta \mathbb{L}_{\geq 0} = \alpha \mathbb{L}_{\geq 0}$ over $\{0\}$.) Since $\mathbb{L}$ is open in $\beta \mathbb{L}$, the compact-open topology on $\homeo_{cpt}(\mathbb{L})$ (induced by $\mathbb{L}$) is coarser than the compact-open topology induced by $\beta X$. In order to prove the converse, let $K \subseteq U \subseteq \beta \mathbb{L}$, where $K$ is compact and $U$ is open, and consider the subbasic neighborhood of the identity
\begin{align}
    W=\{h \in \homeo_{cpt}( \mathbb{L}) \mid h^\beta (K) \subseteq U \},
\end{align}
where $h^\beta\colon \beta \mathbb{L} \rightarrow \beta \mathbb{L}$ is the extension of $h$. Since $K$ is compact, it may be covered by finitely many connected closed subsets $I_1,\ldots,I_n$ of  $\beta \mathbb{L}$ such that $I_j \subseteq U$. For $j=1,\ldots, n$, put
\begin{align}
    W_j:=\{h \in \homeo_{cpt}( \mathbb{L}) \mid h^\beta (I_j ) \subseteq U \}.
\end{align}
One has $\bigcap\limits_{j=1}^n W_j \subseteq W$, and so without loss of generality we may assume that $K$ is connected. We distinguish among three cases.

{\itshape Case 1.} If $\{-\infty,\infty\} \subseteq K$, then $K=U=\beta \mathbb{L}$, and thus $W= \homeo_{cpt}( \mathbb{L})$ is open.

{\itshape Case 2.} If $K \subseteq \mathbb{L}$, then by replacing $U$ with $U \cap \mathbb{L}$, one sees that $W$ is open in the compact-open topology. (Compactly supported homeomorphisms of $\mathbb{L}$ preserve the order, and thus their extensions to $\beta \mathbb{L}$ fix $\infty$ and $-\infty$.)

{\itshape Case 3.} If $|K \cap \{-\infty,\infty\}|=1$, then without loss of generality we may assume that $\infty \in K$ and
$-\infty \notin K$. There are closed connected sets $K_1,K_2\subseteq K$ such that $K=K_1\cup K_2$, 
$K_1 \subseteq \mathbb{L}_{\geq 0} \backslash \{0\}$, and $\infty \notin K_2$. 
Put $U_1:=U\cap \mathbb{L}_{\geq 0} \backslash \{0\}$ and $U_2=U$, and for $i=1,2$, put
\begin{align}
    W_i^\prime:=\{h \in \homeo_{cpt}( \mathbb{L}) \mid h^\beta (K_i) \subseteq U_i \}.
\end{align}
One has $W_1^\prime \cap W_2^\prime \subseteq W$. By Case 2,  $W_2^\prime$ is open in the compact-open topology. Thus, it remains to 
show that $W_1^\prime$ is a neighborhood of the identity in the compact-open topology. Put $K^\prime \! :=K_1 \cup (-K_1)$ and $U^\prime=U_1 \cup (-U_1)$, where $-$ stands for the involution interchanging the two copies of $\alpha \mathbb{L}_{\geq 0}$ in $\beta \mathbb{L}$, and put
\begin{align}
    W^\prime:=\{h \in \homeo_{cpt}( \mathbb{L}) \mid h^\beta (K^\prime) \subseteq U^\prime \}.
\end{align}
If $h \in W^\prime$, then $h^\beta(K_1) \subseteq U_1 \cup (-U_1)$, and the latter union is disjoint, because $U_1\subseteq \mathbb{L}_{\geq 0} \backslash \{0\}$. Thus, the connected set $h^\beta(K_1)$ is contained in $U_1$ in its entirety, because
$h^\beta(\infty) \in h^\beta(K_1) \in U_1$. Therefore, $h \in W^\prime$. This shows that $W^\prime \subseteq W_1^\prime$. 
For $V=\beta \mathbb{L}\backslash K$ and $C=\beta \mathbb{L} \backslash U$, one has
\begin{align}
    W^\prime = \{h \in \homeo_{cpt}( \mathbb{L}) \mid h^\beta (V) \supseteq C \} = 
    \{h^{-1} \in \homeo_{cpt}( \mathbb{L}) \mid h^\beta (C) \subseteq V \}.
\end{align}
Since $C,V\subseteq \mathbb{L}$ and the compact-open topology on $\homeo_{cpt}( \mathbb{L})$ is a group topology, this proves that $W^\prime$ is open in the compact-open topology. 
\end{proof}

Similarly to the proofs found in Section \ref{sect:cofin} and the proof of Theorem~\ref{thm:main:cts}$\empty^\prime$, the proofs of parts (a) and (b) of Theorem~\ref{thm:main:homeo} rely on the tightness of the groups involved.
Unfortunately, this approach cannot be used to prove part (c) of 
Theorem~\ref{thm:main:homeo}, because 
$\homeo_K(X)$ need not be countably tight for a big space $X$.

\begin{example}
Let $\kappa$ be an ordinal, and  $\lambda < \kappa$ a limit ordinal of uncountable cofinality (for example, $\lambda=\omega_1$ and $\kappa=\omega_2$). The set $[0,\lambda]=[0,\lambda+1)$ is clopen in $X=\kappa$, and thus  
\begin{align}
\homeo_{[0,\lambda]}(\kappa) = \homeo([0,\lambda]),
\end{align}
where the latter is equipped with the compact-open topology. We show that 
\begin{align}
t(\homeo([0,\lambda])) \geq \operatorname{cf}(\lambda).
\end{align}
Consider the subset $A:=\bigcup\limits_{\beta < \lambda} \homeo_{[0,\beta]}([0,\lambda])$ of $\homeo([0,\lambda])$. If $B\subseteq A$ such that $|B|<  \operatorname{cf}(\lambda)$, then $\tau:=\sup \bigcup\limits_{f\in B} \operatorname{supp}(f) < \lambda$, and thus $B \subseteq \homeo_{[0,\tau]}([0,\lambda])$; in particular,  $\overline B \subseteq A$. On the other hand, for $h\colon [0,\lambda] \rightarrow [0,\lambda]$ defined by $h(\alpha+1)=h(\alpha+2)$ and $h(\alpha+2)=\alpha+1$ for limit ordinals $\alpha$ and $h(x)=x$ otherwise, one can easily see that $h\notin A$ but $h\in \overline A$.
(For every limit ordinal $\beta<\lambda$, put $h_\beta(x)=h(x)$ if $x \leq \beta$ and $h_\beta(x)=x$
if $x > \beta$. Clearly, $h_\beta\in A$, and $h_\beta \rightarrow h$ in the uniform topology, which coincides with the compact-open topology.)
\end{example}

The proof of part (c) of Theorem~\ref{thm:main:homeo} consists of a number of steps. We first show that CSHP is inherited by clopen subsets, and reduce the proof to the special case of $X=\kappa^n$. We then provide a description of the topology of $\homeo(K)$ for a compact zero-dimensional space $K$, and use it to describe open neighborhoods of the identity in $\homeo_{K} (\kappa^n)$.

\begin{lemma} \label{cshp:lemma:clopen}
Let $X$ be a Tychonoff space and $A\subseteq X$ a clopen subset. 

\begin{enumerate}
    \item 
    $\homeo_{cpt}(A)$ naturally embeds as a closed topological subgroup into $\homeo_{cpt}(X)$; and
    
    \item
    if $X$ has CSHP, then so does $A$.
\end{enumerate}
\end{lemma}

\begin{proof} (a) Since $A$ is clopen,  $\beta A$ embeds as a clopen subset into $\beta X$, and every homeomorphism of $\beta A$ can be extended to a homeomorphism of $\beta X$ by defining it as the identity on $\beta X \backslash \beta A$. Thus, $\homeo(\beta A)$ embeds as a topological subgroup into $\homeo(\beta X)$, where the groups are equipped with the respective compact-open topologies. Therefore, $\homeo_{cpt}(A)$ embeds as a topological subgroup into $\homeo_{cpt}(X)$, and one can identify
\begin{align}
    \homeo_{cpt}(A) = \{f \in \homeo_{cpt}(X) \mid f_{|X\backslash A} = \operatorname{id}_{X\backslash A}\}.
\end{align}
Hence, $\homeo_{cpt}(A)$ is closed in $\homeo_{cpt}(X)$.

(b) Since $X$ has CSHP, by Remark~\ref{remark:cts:acp}, for every closed subgroup $H$ of $\homeo_{cpt}(X)$, one has
\begin{align}
    H = \colim\limits_{K \in \mathscr{K}(X)} (\homeo_K(X)\cap H).
\end{align}
In particular, for $H=\homeo_{cpt}(A)$, one obtains
\begin{align}
     \homeo_{cpt}(A)= \colim\limits_{K \in \mathscr{K}(X)} (\homeo_K(X)\cap \homeo_{cpt}(A))  =
     \colim\limits_{K \in \mathscr{K}(A)} \homeo_K(A),
\end{align}
and hence $A$ has CSHP.
\end{proof}

If $\kappa$ is a cardinal and $\lambda_1,\ldots,\lambda_j < \kappa$ are successor ordinals, then $\kappa^n \times \lambda_1 \times \cdots \times \lambda_j$ is a clopen subset of $\kappa^{n+j}$. Thus, Lemma~\ref{cshp:lemma:clopen} yields the following reduction.

\begin{corollary}  \label{cshp:cor:clopen} 
Let $\kappa$ be a cardinal. If the space $\kappa^n$ has CSHP for every $n$, then $\kappa^n \times \lambda_1 \times \cdots \times \lambda_j$ has CSHP for every $n$ and successor ordinals $\lambda_1,\ldots,\lambda_j < \kappa$. \qed
\end{corollary}

For a Tychonoff space $X$ and a finite clopen partition $\mathcal{A}=\{A_1,\ldots,A_k\}$ of $X$ (i.e., each $A_i$ is clopen), put
\begin{align}
    U(\mathcal{A}):=\{f \in \homeo_{cpt}(X) \mid f(A_i)=A_i \mbox{ for all } i \},
\end{align}
the subgroup of homeomorphisms $f$ such that each $A_i$ is $f$-invariant.

\begin{lemma} \label{cshp:lemma:0dim}
Let $X$ be a zero-dimensional space. Then the family
\begin{align}
    \{U(\mathcal{A}) \mid \mathcal{A} \mbox{ is a finite clopen partition of $X$} \}
\end{align}
is a base at the identity for $\homeo_{cpt}(X)$.
\end{lemma}

\begin{proof}
Since there is a one-to-one correspondence between clopen partitions of $X$ and of $\beta X$, and $\homeo_{cpt}(X)$ is a topological subgroup of $\homeo(\beta X)$ (where the latter is equipped with the compact-open topology), we may assume without loss of generality that $X$ is compact.

If $\mathcal{A}=\{A_1,\ldots,A_k\}$ is a finite clopen partition of $X$, then each $A_i$ is compact, and
\begin{align}
    U(\mathcal{A}) = \bigcap\limits_{i=1}^k \{f \in \homeo_{cpt}(X) \mid f(A_i) \subseteq A_i\}
\end{align}
is an open neighborhood of the identity in $\homeo_{cpt}(X)$. Conversely, consider a subbasic open set $W=\{f\in \homeo_{cpt}(X) \mid f(K) \subseteq U\}$ containing the identity, where $K$ is compact and $U$ is open in~$X$. Since $W$ contains the identity, $K\subseteq U$, and there is a clopen subset $A_1$ such that $K \subseteq A_1 \subseteq U$, because $X$ is zero-dimensional. Thus, for $\mathcal{A}:=\{A_1,X\backslash A_1\}$, one has $U(\mathcal{A})\subseteq W$. Since the collection of sets of the form $U(\mathcal{A})$ is a filter base, this completes the proof.
\end{proof}

\begin{corollary} \label{cshp:cor:discrete}
If a Tychonoff space $X$ contains an infinite discrete clopen subset, then $X$ does not have CSHP. 
\end{corollary}

\begin{proof}
By Lemma~\ref{cshp:lemma:clopen}(b), it suffices to show that if $X$ is an infinite discrete space, then it does not have CSHP. If $X$ is infinite discrete, then $\homeo_K(X)$ is finite for every $K \in \mathscr{K}(X)$, and thus $\colim\limits_{K\in \mathscr{K}(X)} \homeo_K(X)$ is discrete too; however, by Lemma~\ref{cshp:lemma:0dim}, $\homeo_{cpt}(X)$ is not discrete, because $U(\mathcal{A})$ is an infinite group for every finite partition $\mathcal{A}$ of $X$.
\end{proof}

Our next step is to describe open neighborhoods of the identity in $\homeo_K(\kappa^n)=\homeo(K)$ for some special compact open subsets $K \subseteq \kappa^n$. To that end, we describe finite clopen partitions of spaces of the form $\lambda_1 \times \cdots \times \lambda_n$, where $\lambda_i$ are ordinals with $\operatorname{cf}(\lambda_i)\neq \omega$. (Observe that the space $\lambda_1 \times \cdots \times \lambda_n$ is long if and only if $\operatorname{cf}(\lambda_i)\neq \omega$ for every $i$.)
To that end, we introduce the notion of a grid partition.

\begin{definition}
Let $\lambda_1,\ldots,\lambda_n$ be ordinals, and put $X=\lambda_1 \times \cdots \times \lambda_n$.

\begin{enumerate}
    \item 
For a finite subset $F_i=\{x_1 < \cdots < x_N\}\subseteq \lambda_i$, the {\itshape grid partition of $\lambda_i$ with respect to $F_i$} is
\begin{align}
    \mathcal{G}_{F_i}^{\lambda_i}:=\{[0,x_1]\} \cup \{[x_j+1,x_{j+1}] \mid j = 1,\ldots,N-1\} \cup \{[x_N+1,\cdot)\}],
\end{align}
where $[\alpha,\beta]:=\{x \in \lambda_i \mid \alpha \leq x \leq \beta\}$, and $[\alpha,\cdot):=\{x \in \lambda_i \mid \alpha \leq x\}$.

    \item
For finite subsets $F_i \subseteq \lambda_i$ ($i=1,\ldots,n$), the {\itshape grid partition of $X$ with respect to $\{F_i\}_{i=1}^n$} is
\begin{align}
\mathcal{G}_{\{F_i\}_{i=1}^n}^X:= \left\{ \prod\limits_{i=1}^n U_i \mid U_i \in \mathcal{G}_{F_i}^{\lambda_i} \right\}.
\end{align}

\end{enumerate}
\end{definition}

\begin{theorem} \label{cshp:thm:grid}
Let $\lambda_1,\ldots,\lambda_n$ be ordinals with $\operatorname{cf}(\lambda_i)\neq \omega$, and put $X=\lambda_1 \times \cdots \times \lambda_n$. For every finite clopen partition $\mathcal{A}$ of $X$ there are finite subsets $F_i \subseteq \lambda_i$ such that $\mathcal{G}_{\{F_i\}_{i=1}^n}^X$ is a refinement of $\mathcal{A}$.
\end{theorem}

In order to prove Theorem~\ref{cshp:thm:grid}, we first need a lemma.

\begin{lemma} \label{cshp:lemma:boxes}
Let $\lambda_1,\ldots,\lambda_n$ be ordinals, put $X=\lambda_1 \times \cdots \times \lambda_n$, and let $K\subseteq X$ be a non-empty compact subset. Then:

\begin{enumerate}
    \item 
$K$ has a maximal element with respect to the lexicographic order on $X$; and

    \item 
if, in addition, $K$ is open, then $K$ is the disjoint union of finitely many compact open sets of the form $\prod\limits_{i=1}^n [x_i,y_i]$, where for each $i$, either $x_i=0$ or $x_i$ is a successor ordinal.
\end{enumerate}
\end{lemma}

\begin{proof}
Let $\preceq$ denote the lexicographic order on $X$.

(a) Let $\pi_i \colon X \rightarrow \lambda_i$ denote the canonical projection for $i=1,\ldots,n$. We construct $K_1,\ldots,K_{n+1}$ inductively. Put $K_1=K$. Suppose that the compact non-empty sets $K_1,\ldots,K_i$ have already been constructed. The set $\pi_i(K_i)$ is compact and non-empty, and so $a_i:=\max \pi_i(K_i)$ exists. Put 
\begin{align}
K_{i+1}:=\pi_i^{-1}(\{a_i\}) \cap K_i.    
\end{align}
Then $K_{n+1}=\{(a_1,\ldots,a_n)\}$ is a singleton, and $(a_1,\ldots,a_n)=\max(K,\preceq)$.

(b) Let $a=(a_1,\ldots,a_n)\in K$ be the maximal element of $K$ with respect to $\preceq$. Since $K$ is open, there is $x=(x_1,\ldots,x_n)\in X$ such that $a\in \prod\limits_{i=1}^n [x_i,a_i]\subseteq K$, and either $x_i=0$ or $x_i$ is a successor ordinal. The set $U=\prod\limits_{i=1}^n [x_i,a_i]$ is clopen, and so $K^\prime:=K\backslash U$ is again compact open. If $K^\prime=\emptyset$, then we are done. Otherwise, $\max(K^\prime,\preceq) \precneqq \max(K,\preceq)$, and one may repeat the same argument for $K^\prime$. Since each $\lambda_i$ is well-ordered, so is $(X,\preceq)$. Therefore, this process will terminate after finitely many steps.
\end{proof}

We are now ready to prove Theorem~\ref{cshp:thm:grid}.

\begin{proof}[Proof of Theorem~\ref{cshp:thm:grid}.] Let $\mathcal{A}=\{A_1,\ldots,A_m\}$ be a finite clopen partition of $X$. Then 
\begin{align}
    \mathcal{A}^\prime:=\{\operatorname{cl}_{\beta X} A_1,\ldots,\operatorname{cl}_{\beta X} A_m\}
\end{align}
is a partition of $\beta X$ into compact open subsets. 

In order to apply  Lemma~\ref{cshp:lemma:boxes}, we show that $\beta X$ is a product of ordinals. Each space $\lambda_i$ is sequentially compact, because $\operatorname{cf}(\lambda_i)\neq \omega$. Thus, the product $X=\prod\limits_{i=1}^n \lambda_i$ is sequentially compact, and in particular, it is pseudocompact (cf.~\cite[3.10.35, 3.10.30]{Engel}). Therefore, by Glicksberg's Theorem (\cite[Theorem~1]{Glick2}),
\begin{align}
    \beta X = \prod\limits_{i=1}^n \beta \lambda_i.
\end{align}
If $\operatorname{cf}(\lambda_i)=1$, then $\beta \lambda_i=\lambda_i$, and if $\operatorname{cf}(\lambda_i) > \omega$, then $\beta \lambda_i = \lambda_i +1$ (\cite[5N1]{GilJer}). Consequently, $\beta X$ is also a product of ordinals. 

By Lemma~\ref{cshp:lemma:boxes}, each $\operatorname{cl}_{\beta X} A_j$ is the disjoint union of compact open sets of the form $\prod\limits_{i=1}^n [x_i,y_i]$, where for each $i$, either $x_i=0$ or $x_i$ is a successor ordinal. Thus, $\mathcal{A}^\prime$ has a compact open refinement 
\begin{align}
        \mathcal{A}^{\prime\prime}:=\left\{\prod\limits_{i=1}^n [x_i^{(1)},y_i^{(1)}], \ldots, \prod\limits_{i=1}^n [x_i^{(l)},y_i^{(l)}]\right\},
\end{align}
where $x_i^{(j)}=0$ or $x_i^{(j)}$ is a successor ordinal for every $i$ and $j$. Put
\begin{align}
    F_i:=\{y_i^{(j)} \mid j=1,\ldots,l \mbox{ and } y_i^{(j)} < \lambda_i\}.
\end{align}
We claim that the grid partition $\mathcal{G}_{\{F_i\}_{i=1}^n}^{\beta X}$ is a refinement of $\mathcal{A}^{\prime\prime}$. Let $M\in \mathcal{G}_{\{F_i\}_{i=1}^n}^{\beta X}$. Then there are $\alpha=(\alpha_1,\ldots,\alpha_n)$ and $\gamma=(\gamma_1,\ldots,\gamma_n)$ in $\beta X$ such that $M=\prod\limits_{i=1}^n [\alpha_i,\gamma_i]$. Since $\mathcal{A}^{\prime\prime}$ covers $\beta X$, there is $k$ such that $\alpha \in \prod\limits_{i=1}^n [x_i^{(k)},y_i^{(k)}]$. We show that $M \subseteq \prod\limits_{i=1}^n [x_i^{(k)},y_i^{(k)}]$. To that end, it suffices to prove that $\gamma \in \prod\limits_{i=1}^n [x_i^{(k)},y_i^{(k)}]$, or equivalently, that $\gamma_i \leq y_i^{(k)}$ for every $i$. Let $1 \leq i \leq n$. If $\alpha_i=0$, then $\gamma_i = \min F_i$ (or $\gamma_i=\lambda_i$ if $F_i=\emptyset$), and thus $\gamma_i \leq y_i^{(k)}$. Otherwise, $\alpha_i = y_i^{(r)}+1$ for some $r$, and $\gamma_i$ is the next element of $F_i$ (or $\gamma_i=\lambda_i$), and thus again $\gamma_i \leq y_i^{(k)}$. 

This shows that $\mathcal{G}_{\{F_i\}_{i=1}^n}^{\beta X}$ is a refinement of $\mathcal{A}^{\prime\prime}$. Since $F_i \subseteq \lambda_i$ for every $i$, it follows that $\mathcal{G}_{\{F_i\}_{i=1}^n}^{X}$ is a refinement of $\mathcal{A}$.
\end{proof}

The last step before proving part (c) of Theorem~\ref{thm:main:homeo} is a technical proposition that allows one to use grid partitions with a fixed number of points in dealing with open subsets of the colimit space topology on $\colim\limits_{K \in \mathscr{K}(X)} \homeo_K(X)$.

\begin{proposition} \label{cshp:prop:stabilize}
Let $\lambda_1,\ldots,\lambda_n$ be ordinals with $\operatorname{cf}(\lambda_i)\neq \omega$, and put $X=\prod\limits_{i=1}^n \lambda_i$, and let $W$ be an open set containing the identity in the colimit space topology on $\colim\limits_{K \in \mathscr{K}(X)} \homeo_K(X)$. Then there are $m_1,\ldots,m_n\in\mathbb{N}$ and $\alpha^{(1)} \in X$ such that for every $\alpha=(\alpha_1,\ldots,\alpha_n) \geq \alpha^{(1)}$, there are $F_i^\alpha \subseteq \alpha_i$ such that $|F_i^\alpha|=m_i$ and
\begin{align}
    U\left(\mathcal{G}_{\{F_i^\alpha\}_{i=1}^n}^{\downarrow\alpha}\right)  \subseteq W \cap \homeo_{\downarrow\alpha}(X).
\end{align}
\end{proposition}

\begin{proof}
Since $W$ is open in the colimit space topology, the intersection $W \cap \homeo_{\downarrow\alpha}(X)$ is open in $\homeo_{\downarrow\alpha}(X)$ for every $\alpha \in X$. One has
\begin{align}
    \homeo_{\downarrow\alpha}(X) \cong \homeo(\downarrow \alpha)=\homeo_{cpt}(\downarrow \alpha),
\end{align}
because $\downarrow \alpha$ is a compact open subset of the zero-dimensional space $X$. Thus, by Lemma~\ref{cshp:lemma:0dim}, for every $\alpha \in X$ there is a finite clopen partition $\mathcal{A}_\alpha$ of $\downarrow \alpha$ such that $U(\mathcal{A}_\alpha)\subseteq W \cap \homeo_{\downarrow\alpha}(X)$. By Theorem~\ref{cshp:thm:grid} applied to $X=\downarrow \alpha$, each $\mathcal{A}_\alpha$ has a grid partition refinement $\mathcal{G}_{\{E_i^\alpha\}_{i=1}^n}^{\downarrow \alpha}$ with $E_i \subseteq \alpha_i$, and consequently $U(\mathcal{G}_{\{E_i^\alpha\}_{i=1}^n}^{\downarrow \alpha}) \subseteq W \cap \homeo_{\downarrow\alpha}(X)$.

We show now that it is possible to choose the grid points $E_i^\alpha$ in a way that the sizes of $E_i^\alpha$ stabilize. To that end, let $\preceq$ denote the lexicographic order on $\mathbb{N}^n$. For every $\alpha \in X$, put
\begin{align}
    (m_1^\alpha,\ldots,m_n^\alpha):= \min\limits_{\preceq} \left\{(|E_1^\alpha|,\ldots,|E_n^\alpha|) \mid E_i^\alpha \subseteq \alpha_i, \, U\left(\mathcal{G}_{\{E_i^\alpha\}_{i=1}^n}^{\downarrow \alpha}\right) \subseteq W \cap \homeo_{\downarrow\alpha}(X) \right\}.
\end{align}
We show that the map $f\colon X\rightarrow (\mathbb{N}^n,\preceq)$ defined by $\alpha\mapsto (m_1^\alpha,\ldots,m_n^\alpha)$ is monotone. Let $\alpha \leq \beta \in X$, and pick $F_i^\beta \subseteq \beta_i$ such that $|F_i^\beta|=m_i^\beta$ and $U(\mathcal{G}_{\{F_i^\beta\}_{i=1}^n}^{\downarrow \beta}) \subseteq W \cap \homeo_{\downarrow\beta}(X)$. For $E_i^\alpha = F_i^\beta \cap \alpha_i$, one has $U(\mathcal{G}_{\{E_i^\alpha\}_{i=1}^n}^{\downarrow \alpha}) \subseteq W \cap \homeo_{\downarrow\alpha}(X)$ and $|E_i^\alpha|\leq |F_i^\beta|=m_i^\beta$ for every $i$. In particular, 
\begin{align}
    (m_1^\alpha,\ldots,m_n^\alpha) \preceq (|E_1^\alpha|,\ldots,|E_n^\alpha|) \preceq (|F_1^\alpha|,\ldots,|F_n^\alpha|) =(m_1^\beta,\ldots,m_n^\beta).
\end{align}
Since $\operatorname{cf}(\lambda_i)\neq \omega$, the directed set $(X,\leq)$ is long, and thus so is its image $f(X)$ under the order-preserving map $f$. Therefore, the countable set $f(X)$ has an upper bound $(m_1,\ldots,m_n) \in f(X)$,  and it is the maximal element of $f(X)$. Pick $\alpha^{(1)} \in X$ such that $f(\alpha^{(1)})=(m_1,\ldots,m_n)$. Then for every $\alpha \geq \alpha^{(1)}$, one has $m_i^\alpha=m_i$ for every $i$. Hence, one may pick  $F_i^\alpha \subseteq \alpha_i$ such that $|F_i^\alpha|=m_i$ and $U(\mathcal{G}_{\{F_i^\alpha\}_{i=1}^n}^{\downarrow \beta}) \subseteq W \cap \homeo_{\downarrow\alpha}(X)$.
\end{proof}

We turn now to the proof of part (c) of Theorem~\ref{thm:main:homeo}. We will show that when the $\lambda_i$ are equal to a regular uncountable cardinal $\kappa$, then not only do the cardinalities of the $F_i^\alpha$ stabilize, but also the sets themselves do. To do so, recall that a subset $S$ of a cardinal $\kappa$ is called \emph{stationary} if it intersects every closed and unbounded subset of $\kappa$. Fodor's Pressing Down Lemma states that:

\begin{ftheorem}[{\cite[Theorem 8.7]{SetTheoryJech}}] \label{cshp:thm:Fodor}
Let $\kappa$ be a regular uncountable cardinal, and let $S\subseteq \kappa$ be a stationary subset. If $f\colon S\rightarrow \kappa$ satisfies $f(x) <x$ for every $x\in S$, then there is a stationary set $T\subseteq S$ such that $f_{|T}$ is constant.
\end{ftheorem}

\begin{proof}[Proof of Theorem~\ref{thm:main:homeo}(c).]
By Corollary~\ref{cshp:cor:clopen}, it suffices to prove the statement for $X=\kappa^n$, where $\kappa$ is a regular uncountable cardinal. The colimit space topology on 
$\colim\limits_{K \in \mathscr{K}(X)} \homeo_K(X)$ is finer than the topology of $\homeo_{cpt}(X)$. We show the converse.

Let $W$ be an open subset of $\colim\limits_{K \in \mathscr{K}(X)} \homeo_K(X)$ and let $h\in W$. Since translation is continuous in the colimit space topology, $W_0:=Wh^{-1}$ is an open set containing the identity. Thus, by Proposition~\ref{cshp:prop:stabilize}, there are $m_1,\ldots,m_n \in \mathbb{N}$ and $\alpha^{(1)}\in X$ such that for every $\alpha \geq \alpha^{(1)}$ there are $F_i^\alpha \subseteq \alpha$ with $|F_i^\alpha|=m_i$ and 
\begin{align} \label{cshp:eq:W_0}
    U\left(\mathcal{G}_{\{F_i^\alpha\}_{i=1}^n}^{\downarrow\alpha}\right)  \subseteq W_0 \cap \homeo_{\downarrow\alpha}(X).
\end{align}
Let $\Delta\colon \kappa \rightarrow \kappa^n$ denote the diagonal map, and put
$\gamma^{(1)} = \max\limits_{1 \leq i \leq n} \alpha_i^{(1)}$, where $\alpha^{(1)}=(\alpha_1^{(1)},\ldots,\alpha_n^{(1)})$. For $\gamma \geq \gamma^{(1)}$, for  every $(i,j)$ such that $1\leq i \leq n$ and  $1\leq j \leq m_i$, let $f_{i,j}(\gamma)$ denote the $j$-th point in $F_i^{\Delta \gamma}$. Since $F_i^{\Delta \gamma} \subseteq (\Delta \gamma)_i =\gamma$, one has $f_{i,j}(\gamma) < \gamma$. By repeated application of Theorem~\ref{cshp:thm:Fodor}, one obtains a stationary set $T\subseteq [\gamma^{(1)},\kappa)$ such that $f_{i,j|T}$ is constant for every $1\leq i \leq n$ and  $1\leq j \leq m_i$. In other words, there are $F_i \subseteq \kappa$ such that $F_i^{\Delta \gamma} = F_i$ for every $\gamma \in T$. We show that
\begin{align}
    U\left(\mathcal{G}_{\{F_i\}_{i=1}^n}^{X}\right)  \subseteq W_0. 
\end{align}
Let $g \in U(\mathcal{G}_{\{F_i\}_{i=1}^n}^{X})$. Since $g$ has compact support, there is $\beta$ such that $g \in \homeo_{\downarrow \Delta \beta }(X)$. Since $T$ is stationary, in particular, it intersects $[\beta,\kappa)$, and so there is $\gamma \in T$ such that $\beta \leq \gamma$. Therefore,
\begin{align}
g \in U\left(\mathcal{G}_{\{F_i\}_{i=1}^n}^{X}\right) \cap \homeo_{\downarrow \Delta \gamma }(X) = 
U\left(\mathcal{G}_{\{F_i^{\Delta \gamma}\}_{i=1}^n}^{\downarrow \Delta \gamma}\right) \stackrel{(\ref{cshp:eq:W_0})}{\subseteq} W_0.
\end{align}
Hence, $U(\mathcal{G}_{\{F_i\}_{i=1}^n}^{X}) h \subseteq W$, and $W$ is open in the topology of $\homeo_{cpt}(X)$.
\end{proof}

\section*{Acknowledgments}

We would like to express our heartfelt gratitude to Karl H. Hofmann for introducing us to each other, and to the organizers of the 2015 Summer Conference on Topology and its Applications for providing an environment conducive for this collaboration to form. We wish to thank Fr\'ed\'eric Mynard for the valuable correspondence. We are grateful to Karen Kipper for her kind help in proofreading this paper for grammar and punctuation. We are grateful to the anonymous referee for their detailed and helpful suggestions
that have contributed to the articulation and clarity of the manuscript.

{\footnotesize

\bibliography{dahmen-lukacs}

}

\begin{samepage}

\bigskip
\noindent
\begin{tabular}{l @{\hspace{1.6cm}} l}
Rafael Dahmen						    & G\'abor Luk\'acs \\
Department of Mathematics				& Department of Mathematics and Statistics\\
Karlsruhe Institute of Technology		& Dalhousie University\\
D-76128 Karlsruhe					    & Halifax, B3H 3J5, Nova Scotia\\
Germany                			        & Canada\\
{\itshape rafael.dahmen@kit.edu}        & {\itshape lukacs@topgroups.ca}
\end{tabular}

\end{samepage}

\end{document}